\documentclass[a4paper,11pt, twoside]{article}
\usepackage{a4} 
\usepackage{geometry} 
\usepackage{amsmath}
\usepackage{amssymb}
\usepackage{amsthm}
\usepackage{graphicx}
\usepackage{caption,subcaption}
\usepackage{multirow}
\usepackage{xstring}
\usepackage[utf8]{inputenc}
\usepackage{amstext,amsbsy,amsopn,eucal,enumerate}
\usepackage{tikz}
\usepackage{pgfplots}
\usepackage{tabularx}
\usepackage[pagewise]{lineno}%\linenumbers
\usetikzlibrary{trees}

\newcommand{\dd}{\mathrm{d}}

\newcommand{\expn}{\operatorname{e}}

\newcommand{\kernel}{\operatorname{ker}}

\newcommand{\im}{\operatorname{im}}
\newcommand{\diag}{\operatorname{diag}}

\newcommand{\beq}{\begin{equation}}
\newcommand{\eeq}{\end{equation}}
\newcommand {\mat}      [1] {\left[\begin{array}{#1}}
\newcommand {\rix}          {\end{array}\right]}
\newcommand {\smat}      [1] {\left[\begin{smallmatrix}{#1}}
\newcommand {\srix}          {\end{smallmatrix}\right]}
\newcommand {\s}      [1] {\begin{smallmatrix}{#1}}
\newcommand {\se}          {\end{smallmatrix}}
\newcommand{\trace}{\operatorname{tr}}
\newcommand{\parent}{\operatorname{parent}}
\newcommand{\Vol}{\operatorname{Vol}}
\newcommand{\leaves}{\operatorname{leaves}}

\newtheorem{defn}{Definition}[section]
\newtheorem{remark}[defn]{Remark}
\newtheorem{example}[defn]{Example}
\newtheorem{lem}[defn]{Lemma}
\newtheorem{prop}[defn]{Proposition} %Proposition entspricht Satz
\newtheorem{kor}[defn]{Corollary}
\newtheorem{thm}[defn]{Theorem}

\title{Tree-based solution representations for quadratic bilinear systems and their consequences in model order reduction}

\author{Martin Redmann\thanks{University of Rostock, Institute of Mathematics, Ulmenstraße 69, 18057 Rostock, Germany, Email: {\tt
martin.redmann@uni-rostock.de}.}
}

\begin{document}

\maketitle

\begin{abstract}
We investigate quadratic bilinear systems by developing novel tree-based representations of their solutions. The proposed framework decomposes the solution into a sequence of coupled bilinear subsystems whose components admit explicit expansions indexed by full binary trees. These representations yield sufficient conditions for the existence of global solutions and lead to new output bounds in terms of reachability Gramians. Motivated by these estimates, we introduce time-limited and infinite-horizon reachability and observability Gramians, establish sufficient conditions for their existence, and characterize them through nonlinear matrix equations. The associated Gramians are employed to identify dominant state-spaces and to derive exact reduced-order models obtained by removing Gramian kernels. Building on these results, we develop a balanced truncation method for quadratic bilinear systems and prove an error bound for the reduced-order approximation. The proposed framework provides a unified connection between tree-based solution representations, nonlinear Gramian theory, and balanced truncation for quadratic bilinear systems, closing several theoretical gaps in the analysis of Gramian-based model reduction for this class of systems.
\end{abstract}



\pagestyle{myheadings}
\thispagestyle{plain}
\markboth{M. Redmann}{Tree-based solution representations for quadratic bilinear systems}

%%%%%%%%%%%%%%%%%%%%%%

\section{Introduction}

Large-scale dynamical systems arise naturally in many areas of science and applied mathematics. While these models often provide an accurate description of the underlying dynamics, their high dimensionality renders numerical simulation and analysis computationally demanding. Model order reduction (MOR) seeks to overcome this difficulty by replacing the original model with a surrogate of significantly smaller dimension that accurately reproduces the essential features of the original system. Most projection-based MOR techniques achieve this by identifying low-dimensional subspaces that capture the dominant dynamics and projecting the governing equations onto these subspaces.

Over the past decades, a variety of MOR methodologies has been developed. Proper orthogonal decomposition \cite{pod} constructs reduced spaces from representative solution trajectories and has become one of the standard tools in data-driven model reduction. Projection methods based on interpolation or $\mathcal{H}_2$-optimality, such as the iterative rational Krylov algorithm \cite{irka}, provide another successful class of reduction techniques. Among system-theoretic approaches, balanced truncation \cite{moo1981} occupies a prominent position due to its robustness and strong mathematical foundation. By employing reachability and observability Gramians, balanced truncation identifies state variables that are simultaneously difficult to control and observe. Moreover, for linear deterministic systems, it preserves stability and is accompanied by rigorous a priori error bounds \cite{BT_bound_enns,BT_bound_glover}.

Extending these favorable properties to nonlinear systems has attracted considerable attention in recent years. Data-driven MOR techniques for nonlinear systems include, for instance, \cite{gosea_antoulas,pod_kramer,qian_data}, whereas interpolation- and optimization-based methods have been investigated in \cite{breiten_benner2,nonlinear_irka}. Several balancing-related approaches have also been proposed for nonlinear systems \cite{nonlinear_bt,Scherpen,verriest_nonlinear_bt}. However, unlike the classical linear theory, these methods generally lack rigorous error estimates. Balancing related MOR techniques for nonlinear system with error bounds have been developed in more recent works, see \cite{incr_bt,gen_incr_bt, redstochbil, nonlinear_drift,  stoch_nonlinear_BT}. These results rely on the assumption of having a system with global one-sided Lipschitz vector fields. This assumption does not hold when quadratic terms are involved. MOR techniques have also been developed for lifted quadratic bilinear representations of nonlinear systems in several works, see for instance \cite{morBenG24, nonlinear_irka, Kramer2022}.

Besides arising directly in the discretization of various nonlinear partial differential equations, quadratic bilinear systems frequently result from polynomial reformulations or lifting procedures applied to more general nonlinear systems. Therefore, their analysis and dimension reduction deserves further attention. Consequently, the development of mathematically rigorous MOR techniques for quadratic bilinear systems is of interest far beyond this particular class of models. Gramian-based MOR has therefore been proposed for quadratic bilinear systems, where generalized reachability and observability Gramians are defined by  Volterra kernels \cite{morBenG24, nonlinear_irka}. Moreover, interpolation- and moment-matching methods for these class of systems have been established in \cite{benner_gug_werner, quadr_bil_moment_matching}. Despite promising numerical results, however, the theoretical understanding of these approaches remains substantially less complete than in the linear setting. In particular, the dynamical origin of the proposed Gramians has remained largely unexplained, sufficient conditions for their existence are only partially understood, and the connection between the Gramians and dominant state-space directions has not been rigorously established. Most importantly, existing balanced truncation methods for quadratic bilinear systems do not provide rigorous input-output error bounds comparable to those available for linear systems.

The purpose of the present work is to provide a comprehensive theoretical foundation for balanced truncation of quadratic bilinear systems. Our analysis begins with a novel tree-based representation of the system solutions, in which the nonlinear dynamics are decomposed into a hierarchy of coupled bilinear subsystems indexed by full binary trees. This representation yields sufficient conditions for the existence of global solutions and leads to new estimates of the system output. Motivated by these estimates, we introduce time-limited and infinite-horizon reachability and observability Gramians together with sufficient conditions for their existence. We further show that these Gramians satisfy nonlinear matrix equations, thereby providing a rigorous analytical interpretation of the algebraic equations underlying previously proposed balancing techniques.

Building upon this Gramian framework, we establish precise links between the eigenspaces of the Gramians and dominant state-space directions, thereby providing a rigorous justification for Gramian-based state truncation. We additionally prove that the kernels of the Gramians can be removed without introducing any approximation error, yielding exact reduced-order models before balancing is performed. Finally, we develop a balanced truncation procedure for quadratic bilinear systems and derive an a priori error bound for the resulting reduced-order models. Altogether, the presented theory unifies tree-based solution representations, nonlinear Gramian theory, associated matrix equations, dominant subspace identification, and balanced truncation within a single analytical framework, thereby closing several theoretical gaps in the current understanding of Gramian-based model reduction for quadratic bilinear systems.

The remainder of the paper is organized as follows. Section~$2$ develops the tree-based solution representation of quadratic bilinear systems and derives associated output estimates. Based on these results, reachability Gramians are introduced, sufficient conditions for their existence are established, and their characterization by nonlinear matrix equations is derived. Section~$3$ develops the corresponding observability theory, investigates dominant state-space directions, and constructs exact reduced-order models obtained by removing Gramian kernels. Subsequently, the balanced truncation procedure is presented, and the main theoretical result of the paper, namely an a priori error bound for the reduced-order approximation, is established.

\section{Quadratic bilinear systems, binary trees and tree-based solution representations}\label{sec2}

In this paper, we study the following quadratic bilinear system
 \begin{subequations}\label{original_system}
\begin{align}\label{stochstatenew}
             \dot x(t)&=Ax(t)+Bu(t)+ N (\mathfrak u(t)\otimes x(t))+H(x(t)\otimes x(t)),\quad x(0) = 0,\\ \label{output_eq}
            y(t) &= Cx(t),\quad t\in [0, T],
\end{align}
\end{subequations}
with $A \in\mathbb R^{n\times n}$,  $B\in  \mathbb R^{n\times m}$, $C\in \mathbb R^{p\times n}$, $N:=\begin{bmatrix}N_1 & \ldots & N_m\end{bmatrix}$,  $H:=\begin{bmatrix}H_1 & \ldots & H_n\end{bmatrix}$, where $N_k, H_j\in\mathbb R^{n\times n}$, $k\in\{1, \dots, m\}$ and $j\in\{1, \dots, n\}$. The controls $u, \mathfrak u\in L^2_T$ take values in $\mathbb R^m$ and satisfy $\|u\|_{L^2_T}^2:=\int_0^T \|u(t)\|^2 dt<\infty$ as well as $\|\mathfrak u\|_{L^2_T}<\infty$  for a given $T>0$, where $\|\cdot\|$ denotes the Euclidean norm and $\langle \cdot, \cdot \rangle$ the corresponding inner product. Further note that $\cdot \otimes \cdot$ denotes the Kronecker product between two matrices/vectors.

\subsection{Series representations for global solutions}

The first goal is to represent the solution $x$ of \eqref{stochstatenew} by $\sum_{i=1}^\infty x_i$, where $x_1$ solves the bilinear system \begin{align}\label{eq_x1}
  \dot x_1(t)=Ax_1(t)+Bu(t)+N (\mathfrak u(t)\otimes x_{1}(t)),\quad x_1(0) = 0                                                                                                                                                                                                                                                                                    \end{align}
and the other $x_i$, $i\geq 2$, result from the following coupled bilinear systems \begin{align}\label{eq_xi}
             \dot x_i(t)=Ax_i(t)+ N (\mathfrak u(t)\otimes x_{i}(t))%+H\Big(\sum_{j_i, j_2=1 \atop j_1+j_2=i}^{i-1} x_{j_1}(t)\otimes x_{j_2}(t)\Big),\quad x_i(0) = 0.
             +H\Big(\sum_{k=1}^{i-1} x_{k}(t)\otimes x_{i-k}(t)\Big),\quad x_i(0) = 0.
\end{align}
This idea differs from the approach used in \cite{morBenG24, nonlinear_irka, benner_gug_werner}, where quadratic bilinear systems are decomposed into a sequence of coupled linear subsystems.
We introduce tree-based expansions for each $x_i$ yielding a representation for $\sum_{i=1}^\infty x_i$. As a consequence, we obtain criteria for the boundedness of $\sum_{i=1}^\infty \|x_i\|$ on $[0, T]$. The following proposition states that this property ensures that $x=\sum_{i=1}^\infty x_i$ is the global solution of \eqref{stochstatenew}. This aspect remains open in previous works on quadratic bilinear systems \cite{morBenG24, nonlinear_irka}.
\begin{prop}\label{prop_global_sol}
Given that $\sum_{i=1}^\infty \|x_i(t)\|\leq c<\infty$ for all $t\in [0, T]$ and given $u, \mathfrak u\in L^2_T$, then $\sum_{i=1}^\infty x_i(t)$,  $t\in [0, T]$, solves \eqref{stochstatenew} globally.
\end{prop}
\begin{proof}
First, let us rewrite the quadratic term as \begin{align*}
 H\Big(\sum_{k=1}^{i-1} x_{k}(t)\otimes x_{i-k}(t)\Big) = H\Big(\sum_{j_i, j_2\geq 1 \atop j_1+j_2=i} x_{j_1}(t)\otimes x_{j_2}(t)\Big).
\end{align*}
For a given index $\mathfrak n$, let us sum the equations of the $x_i$'s up to $\mathfrak{n}$. This leads to \begin{align}\label{nthiterate}
    \sum_{i=1}^\mathfrak n x_i(t)=\int_0^t A \sum_{i=1}^\mathfrak n x_i(s)+B u(s)+ N (\mathfrak u(s)\otimes \sum_{i=1}^\mathfrak n x_{i}(s))+H\Big(\sum_{j_i, j_2\geq 1 \atop j_1+j_2\leq \mathfrak n} x_{j_1}(s)\otimes x_{j_2}(s)\Big)\dd s
\end{align}
for $t\in [0, T]$. By assumption, we obtain \begin{align*}
   & \|A \sum_{i=1}^\mathfrak n x_i(s)+B u(s)+ N (\mathfrak u(s)\otimes \sum_{i=1}^\mathfrak n x_{i}(s))+H\Big(\sum_{j_i, j_2\geq 1 \atop j_1+j_2\leq \mathfrak n} x_{j_1}(s)\otimes x_{j_2}(s)\Big)\|\\
  &\leq \|A\| \sum_{i=1}^\mathfrak n \|x_i(s)\|+\|B\| \|u(s)\|+ \|N\|\|\mathfrak u(s)\|\sum_{i=1}^\mathfrak n \|x_{i}(s)\|+\|H\|\sum_{j_i, j_2\geq 1 \atop j_1+j_2\leq \mathfrak n} \|x_{j_1}(s)\| \|x_{j_2}(s)\|\\
  &\leq \|A\| c+\|B\| \|u(s)\|+ \|N\|\|\mathfrak u(s)\|c+\|H\|c^2
\end{align*}
for all $s\in[0, T]$. This upper bound is integrable, so that Lebesgue's theorem can be applied. We can take the limit in \eqref{nthiterate}  as $\mathfrak n\to\infty$ and obtain the result of this proposition.
\end{proof}

\subsection{Binary trees and tree-based solution representations}
We consider rooted trees and, in particular, study full binary trees below. These are trees in which every node has either $0$ or $2$ children. Here, the children of a node $v$ are the nodes directly connected to $v$ that are one level farther from the root. Moreover, the nodes without children are called leaves.
The sets $\mathcal T_i$ of binary trees with $i$ leaves can now be formally defined as follows:
\begin{equation}\label{def_bin_trees}
\begin{aligned}
 \mathcal T_1= \{\bullet\}, \quad \mathcal T_i &=  \{\tau= [\tau_1, \tau_2]: \tau_1 \in\mathcal T_k, \tau_2\in \mathcal T_{i-k}, \text{ where }k=1, \dots, i-1\}    \\
 &=  \{\tau= [\tau_1, \tau_2]: \vert \tau \vert= 2i-1, \text{ and  } \tau \text{ has } i \text{ leaves}\},\quad i\geq 2.\end{aligned}
 \end{equation}
The bracket operation $[\tau_1,\dots, \tau_k]$ forms a new rooted tree by grafting the rooted trees $\tau_1, \dots, \tau_k$  onto a new common root, e.g., $[\bullet, \bullet]= \begin{tikzpicture}[
  grow=up,
  level distance=0.2cm,
  sibling distance=0.4cm,
  every node/.style={circle,draw,fill,inner sep=1.2pt}
]
\node {}
  child { node {} }
  child { node {} };
\end{tikzpicture}$.
\begin{example}
 For illustration purposes, we state $\mathcal T_2$, $\mathcal T_3$ and $\mathcal T_4$. These are given by
\begin{align*}
\mathcal T_2 = \bigg\{\begin{tikzpicture}[
  grow=up,
  level distance=0.2cm,
  sibling distance=0.4cm,
  every node/.style={circle,draw,fill,inner sep=1.2pt}
]
\node {}
  child { node {} }
  child { node {} };
\end{tikzpicture}\bigg\}, \quad
\mathcal T_3=
\bigg\{
\begin{tikzpicture}[
  baseline=(root.base),
  grow=up,
  level distance=0.2cm,
  sibling distance=0.4cm,
  every node/.style={circle,draw,fill,inner sep=1.2pt}
]
\node (root) {}
  child {
    node {}
      child {node {}}
      child {node {}}
  }
  child {node {}};
\end{tikzpicture},\,
\begin{tikzpicture}[
  baseline=(root.base),
  grow=up,
  level distance=0.2cm,
  sibling distance=0.4cm,
  every node/.style={circle,draw,fill,inner sep=1.2pt}
]
\node (root) {}
  child {node {}}
  child {
    node {}
      child {node {}}
      child {node {}}
  };
\end{tikzpicture}
\bigg\},\quad
\mathcal T_4=
\bigg\{\begin{tikzpicture}[grow=up,
  level distance=0.2cm,
  sibling distance=0.4cm,
  every node/.style={circle,draw,fill,inner sep=1.2pt}]
\node {}
  child {
    node {}
      child {
        node {}
          child { node {} }
          child { node {} }
      }
      child { node {} }
  }
  child { node {} };
\end{tikzpicture},\,
\begin{tikzpicture}[grow=up,
  level distance=0.2cm,
  sibling distance=0.4cm,
  every node/.style={circle,draw,fill,inner sep=1.2pt}]
\node {}
  child {
    node {}
      child { node {} }
      child {
        node {}
          child { node {} }
          child { node {} }
      }
  }
  child { node {} };
\end{tikzpicture},\,
\begin{tikzpicture}[
  grow=up,
  level distance=0.2cm,
  level 1/.style={sibling distance=0.6cm},
  level 2/.style={sibling distance=0.4cm},
  every node/.style={circle,draw,fill,inner sep=1.2pt}
]
\node {}
  child {
    node {}
      child { node {} }
      child { node {} }
  }
  child {
    node {}
      child { node {} }
      child { node {} }
  };
\end{tikzpicture},\,
\begin{tikzpicture}[baseline=(root.base), grow=up,
  level distance=0.2cm,
  sibling distance=0.4cm,
  every node/.style={circle,draw,fill,inner sep=1.2pt}]
\node (root) {}
  child { node {} }
  child {
    node {}
      child {
        node {}
          child { node {} }
          child { node {} }
      }
      child { node {} }
  };
\end{tikzpicture},\,
\begin{tikzpicture}[baseline=(root.base), grow=up,
  level distance=0.2cm,
  sibling distance=0.4cm,
  every node/.style={circle,draw,fill,inner sep=1.2pt}]
\node (root) {}
  child { node {} }
  child {
    node {}
      child { node {} }
      child {
        node {}
          child { node {} }
          child { node {} }
      }
  };
\end{tikzpicture}
\bigg\}.
\end{align*}
\end{example}
Let $c_i := \vert\mathcal T_i\vert$ denote the cardinalities of $\mathcal T_i$. By the recursive construction above, the numbers $(c_i)$
 satisfy\begin{align}   \label{catalan_number_recursion}                                                                                                                                                         c_i=\sum_{k=1}^{i-1} c_k c_{i-k}, \quad i\geq 2, \quad c_1=1.                                                                                                                                                            \end{align}
The solution is given by the Catalan numbers, so that $c_i = \frac{1}{i} \binom{2(i-1)}{i-1}$.
Next, we associate a variable $t_v\in[0,T]$ with each node $v$ of a tree $\tau$. For every non-root node $v$, let $\parent(v)$ denote its unique parent, that is, the unique node immediately preceding $v$ on the path to the root. We define the domain of a tree $\tau$ by \begin{align*}
 D(\tau; t) = \{(t_v)_{v\in\tau}\in\mathbb R^{\vert\tau\vert}: 0\leq t_v\leq t\text{ and } t_v\leq t_{\parent(v)}\ \text{for every non-root node }v\in\tau\},                                                                                                                                                                                                                                                                                                                                                                                                                   \end{align*}
$t\in[0, T]$, where $\vert\tau\vert$ is the number of nodes of $\tau$. The volume  of this tree domain (w.r.t. the Lebesgue measure on $\mathbb R^{\vert\tau\vert}$) is given by
\begin{align}\label{vol_D_tau}
\Vol(D(\tau; t))= \frac{t^{\vert\tau\vert}}{\gamma(\tau)},
\end{align}
which can be shown via an induction proof. In \eqref{vol_D_tau} the tree factorial $\gamma$ naturally occurs that is defined by \begin{align*}
 \gamma(\bullet) = 1,\quad \gamma([\tau_1,\ldots,\tau_k]) = |[\tau_1,\ldots,\tau_k]| \prod_{i=1}^k \gamma(\tau_i).
\end{align*}
Below, we use full binary trees to represent $x_i$, $i\geq 1$, satisfying \eqref{eq_x1}  and \eqref{eq_xi}. First, we require the notion of fundamental solutions of \eqref{eq_x1}, i.e., the bilinear part of \eqref{stochstatenew}. Further, the stochastic fundamental solution is introduced, where $\mathfrak{u}$ is white noise.
\begin{defn}\label{defn_fund}
Given that $s\leq t$, we introduce the fundamental solution $\Phi$  of the bilinear part  as the solution of
 \begin{align*}
 \Phi(t,s) = I +\int_s^t A \Phi(v,s) \dd v + \int_s^t N \big(\mathfrak u(v)\otimes \Phi(v,s)\big) \dd v,
\end{align*}
where $I$ is the identity matrix. Let $w$ be an $m$-dimensional Wiener process with independent entries. The fundamental solution is denoted by $\Phi_w$ if $\mathfrak{u}(v)=\dd w(v)/\dd v$ and $s=0$, i.e.,   \begin{align}\label{stoch_fund}
 \Phi_w(t) = I +\int_0^t A \Phi_w(v) \dd v + \int_0^t N \big(\dd w(v)\otimes \Phi_w(v)\big).
\end{align}
\end{defn}
The fundamental solution is crucial as it allows to write \begin{align*}
 x_1(t)= \int_0^t \Phi(t, s) B u(s) \dd s,\quad
 x_i(t)=\sum_{k=1}^{i-1}\int_0^t \Phi(t, s) H\big (x_k(s)\otimes x_{i-k}(s) \big)\dd s,\quad i\geq 2.
 \end{align*}
However, notice that $\Phi$ is control dependent.
These representations can be shown using the product rule and exploiting that $\Phi(t, s)= \Phi(t, 0) \Phi(s, 0)^{-1}$.
We are now ready to find representations for the state variables $x_i$.
\begin{thm}\label{tree_rep}
Let $x_i$, $i\geq 1$, be the solutions of \eqref{eq_x1}  and \eqref{eq_xi}. Then, it holds that \begin{align*}
 x_i(t) = \sum_{\tau\in\mathcal T_i} \int_{D(\tau; t)} F_{\tau}(t, (t_v)_{v\in\tau}) \big(\bigotimes_{\ell\in\leaves(\tau)} u(t_\ell)\big)\dd(t_v)_{v\in\tau},\quad t\in[0, T],
               \end{align*}
where we set \begin{align*}
F_{\bullet}(t, t_\ell) = \Phi(t, t_\ell) B
             \end{align*}
and for a binary tree $\tau=[\tau_1, \tau_2]_{v_0}\in\mathcal T_i$, $i\geq 2$, with root $v_0$, we define \begin{align*}
F_{\tau}(t, t_{v_0}, (t_{v_1})_{v_1\in\tau_1}, (t_{v_2})_{v_2\in\tau_2}) = \Phi(t, t_{v_0}) H\big(F_{\tau_1}(t_{v_0}, (t_{v_1})_{v_1\in\tau_1})\otimes F_{\tau_2}(t_{v_0}, (t_{v_2})_{v_2\in\tau_2})\big).
             \end{align*}
\end{thm}
\begin{proof}
The claim of this theorem is true for $x_1$, since \begin{align*}
 x_1(t)= \int_0^t \Phi(t, s) B u(s) \dd s.                                                   \end{align*}
Let us now assume that the statement is true for $x_1, x_2, \dots, x_{i-1}$. Using the solution representation of $x_i$, $i\geq 2$, based on the fundamental solution in Definition \ref{defn_fund} and the induction hypothesis, we obtain \begin{align*}
 &x_i(t)=\sum_{k=1}^{i-1}\int_0^t \Phi(t, t_{v_0}) H\big (x_k(t_{v_0})\otimes x_{i-k}(t_{v_0}) \big)\dd t_{v_0}\\
 &=\sum_{k=1}^{i-1}\int_0^t \Phi(t, t_{v_0}) H\bigg (
 \sum_{\tau_1\in\mathcal T_k} \int_{D(\tau_1; t_{v_0})} F_{\tau_1}(t_{v_0}, (t_{v_1})_{v_1\in\tau_1}) \big(\bigotimes_{\ell_1\in\leaves(\tau_1)} u(t_{\ell_1})\big)\dd(t_{v_1})_{v_1\in\tau_1}\\
 &\quad
 \otimes \sum_{\tau_2\in\mathcal T_{i-k}} \int_{D(\tau_2; t_{v_0})} F_{\tau_2}(t_{v_0}, (t_{v_2})_{v_2\in\tau_2}) \big(\bigotimes_{\ell_2\in\leaves(\tau_2)} u(t_{\ell_2})\big)\dd(t_{v_2})_{v_2\in\tau_2}\bigg)\dd t_{v_0}   \\
 &=\sum_{k=1}^{i-1}\sum_{\tau_1\in\mathcal T_k}\sum_{\tau_2\in\mathcal T_{i-k}} \int_0^t \int_{D(\tau_2; t_{v_0})} \int_{D(\tau_1; t_{v_0})} \Phi(t, t_{v_0}) H\bigg (
  F_{\tau_1}(t_{v_0}, (t_{v_1})_{v_1\in\tau_1})\otimes  F_{\tau_2}(t_{v_0}, (t_{v_2})_{v_2\in\tau_2})\bigg)\\
 &\quad
  \bigg(\bigotimes_{\ell_1\in\leaves(\tau_1)} u(t_{\ell_1})\otimes \bigotimes_{\ell_2\in\leaves(\tau_2)} u(t_{\ell_2})\big)\bigg) \dd(t_{v_1})_{v_1\in\tau_1}\dd(t_{v_2})_{v_2\in\tau_2}\dd t_{v_0}.                                                                                                                                                      \end{align*}
We define the tree $\tau=[\tau_1, \tau_2]_{v_0}$ with root ${v_0}$. We observe that the leaves of $\tau$ are the union of the leaves of $\tau_1$ and $\tau_2$. Further, we have that \begin{align*}
D(\tau; t)=\{(t_{v_0}, (t_{v_1})_{v_1\in\tau_1}, (t_{v_2})_{v_2\in\tau_2}): 0\leq t_{v_0}\leq t, (t_{v_1})_{v_1\in\tau_1}\in D(\tau_1; t_{v_0}), (t_{v_2})_{v_2\in\tau_2}\in D(\tau_2; t_{v_0})\}.                                                                                                                                                                             \end{align*}
By construction of the set $\mathcal T_i$, we see that $\sum_{k=1}^{i-1}\sum_{\tau_1\in\mathcal T_k}\sum_{\tau_2\in\mathcal T_{i-k}}\dots = \sum_{\tau\in\mathcal T_{i}}\dots$. Therefore, we have \begin{align*}
 x_i(t) = \sum_{\tau\in\mathcal T_i} \int_{D(\tau; t)} F_{\tau}(t, (t_v)_{v\in\tau}) \big(\bigotimes_{\ell\in\leaves(\tau)} u(t_\ell)\big)\dd(t_v)_{v\in\tau}
               \end{align*}
exploiting the definition of $F$.
This yields the claim of the theorem.
\end{proof}
Theorem \ref{tree_rep} is the foundation for finding a solution representation of the state $x$ similar to a Volterra series. Solution representations of Volterra-type are the basis for previous model reduction approaches for quadratic bilinear system \cite{morBenG24, nonlinear_irka, benner_gug_werner}. However, such representations could not be established so far. The result of Theorem \ref{tree_rep} is crucial for the output bound in the next section.

\subsection{An output bound for quadratic bilinear systems based on reachability Gramians}
An application of the representation in Theorem \ref{tree_rep} is the following (output) bound.
\begin{kor}\label{kor_first_bound}
Let $C\in\mathbb R^{p\times n}$. Then, for the solutions of \eqref{eq_x1}  and \eqref{eq_xi} holds that
\begin{align*}
\left\|C x_i(t)\right\|\leq \sqrt{\sum_{\tau\in\mathcal T_i}\int_{D(\tau; t)} \left\|C F_{\tau}(t, (t_v)_{v\in\tau})\right\|_F^2 \dd(t_v)_{v\in\tau}} \sqrt{\sum_{\tau\in\mathcal T_i}\int_{D(\tau; t)}\prod_{\ell\in\leaves(\tau)}\|u(t_\ell)\|^2\dd(t_v)_{v\in\tau}}.
\end{align*}
\end{kor}
\begin{proof}
Based on the inequality of Cauchy-Schwarz and the result of Theorem \ref{tree_rep}, we obtain that
\begin{align*}
\left\|C x_i(t)\right\|&\leq \sum_{\tau\in\mathcal T_i} \int_{D(\tau; t)} \left\|C F_{\tau}(t, (t_v)_{v\in\tau})\right\|_F \Big\|\bigotimes_{\ell\in\leaves(\tau)} u(t_\ell)\Big\|\dd(t_v)_{v\in\tau}\\
&\leq \sum_{\tau\in\mathcal T_i} \sqrt{\int_{D(\tau; t)} \left\|C F_{\tau}(t, (t_v)_{v\in\tau})\right\|_F^2 \dd(t_v)_{v\in\tau}} \sqrt{\int_{D(\tau; t)}\Big\|\bigotimes_{\ell\in\leaves(\tau)} u(t_\ell)\Big\|^2\dd(t_v)_{v\in\tau}}\\
&=\sum_{\tau\in\mathcal T_i} \sqrt{\int_{D(\tau; t)} \left\|C F_{\tau}(t, (t_v)_{v\in\tau})\right\|_F^2 \dd(t_v)_{v\in\tau}} \sqrt{\int_{D(\tau; t)}\prod_{\ell\in\leaves(\tau)}\|u(t_\ell)\|^2\dd(t_v)_{v\in\tau}}\\
&\leq \sqrt{\sum_{\tau\in\mathcal T_i}\int_{D(\tau; t)} \left\|C F_{\tau}(t, (t_v)_{v\in\tau})\right\|_F^2 \dd(t_v)_{v\in\tau}} \sqrt{\sum_{\tau\in\mathcal T_i}\int_{D(\tau; t)}\prod_{\ell\in\leaves(\tau)}\|u(t_\ell)\|^2\dd(t_v)_{v\in\tau}}.
\end{align*}
This concludes the proof.
\end{proof}
\begin{remark}
 It is only reasonable to apply the estimate in Corollary \ref{kor_first_bound} for $i\geq 2$ if $H\neq 0$. Given that $H=0$, we are in the bilinear case, where $x(t)=x_1(t)=\int_0^t \Phi(t, s) B u(s) \dd s$ and $x_i\equiv 0$ for all $i\geq 2$.
\end{remark}
The definition of the Frobenius norm and the properties of the trace yield \begin{align}\label{rel_via_Frob}
\sum_{\tau\in\mathcal T_i} \int_{D(\tau; t)} \left\|C F_{\tau}(t, (t_v)_{v\in\tau})\right\|_F^2 \dd(t_v)_{v\in\tau}= \trace\Big(C \sum_{\tau\in\mathcal T_i}\int_{D(\tau; t)} F_{\tau}(t, (t_v)_{v\in\tau})F_{\tau}(t, (t_v)_{v\in\tau})^\top \dd(t_v)_{v\in\tau} C^\top\Big).                                                                                 \end{align}
Below, we provide a theorem containing an upper bound for the matrix-valued function \\$\sum_{\tau\in\mathcal T_i}\int_{D(\tau; t)} F_{\tau}(t, (t_v)_{v\in\tau})F_{\tau}(t, (t_v)_{v\in\tau})^\top \dd(t_v)_{v\in\tau}$ independent of $t$. This allows to uniformly bound $Cx_i$ in $t$ based on the result of Corollary \ref{kor_first_bound}. In the following, we say that a symmetric matrix $\bar M$ bounds another symmetric matrix $M$ from above if $\bar M-M$ is positive semidefinite. Then, we write $M\leq \bar M$. Before we focus on the above mentioned theorem, we need a lemma on a bound for the fundamental solution of the bilinear part. In that context, let us introduce a Lyapunov operator by \begin{align}\label{lyap_op}
\mathcal L(X):= A X + X A^\top + N (I\otimes X) N^\top,% \quad \bar{\mathcal L}(X):= A_{\bar c} X + X A_{\bar c}^\top + N (I\otimes X) N^\top,
               \end{align}
where %$A_{\bar c}:=A+I \bar c^2/2$ with $\bar c>0$ and
$X\in\mathbb R^{n\times n}$.
\begin{lem}\label{fund_est}
Let $\Phi$ be the fundamental solution according to Definition \ref{defn_fund} and $M$ a positive semidefinite matrix. Then, it holds that
 \begin{align}\label{est_phi}
\Phi(t, s) M \Phi(t, s)^\top \leq  \exp\left\{\int_{s}^t \left\|\mathfrak u(v)\right\|^2 \mathrm{d}v\right\}  Z(t-s),\quad s\leq t\leq T,
\end{align}
where $ Z(t)$, $t\in [0, T]$, satisfies the matrix differential equation
\begin{align}\label{eqZ}
 \dot {Z}(t) = \mathcal L\big(Z(t)\big) ,\quad Z(0) = M.
\end{align}
% If there further is a constant $\bar c>0$, so that $\left\|\mathfrak u(v)\right\|\leq \bar c$ for all $v\in[0, T]$, we have   \begin{align}\label{est_phi2}
% \Phi(t, s) M \Phi(t, s)^\top \leq    \bar Z(t-s),\quad s\leq t\leq T,
% \end{align}
% with \begin{align}\label{eqbarZ}
%  \dot {\bar Z}(t) = \bar{\mathcal L}\big(\bar Z(t)\big),\quad \bar Z(0) = M.
% \end{align}
Moreover, the solution of \eqref{eqZ} can be represented by $Z(t)=\mathbb E[\Phi_w(t)M\Phi_w(t)^\top]$, where $\Phi_w$ is defined through \eqref{stoch_fund}.
\end{lem}
\begin{proof}
The estimate in \eqref{est_phi} was proved in \cite{h2_bil}. The stochastic solution representation of $Z$ can be found in  \cite{redmannbenner} or  \cite{damm}.
% Exploiting that $\mathfrak u$ is bounded by $\bar c$, we obtain from \eqref{est_phi} that  $\Phi(t, s) M \Phi(t, s)^\top \leq  \exp\left\{\bar c^2(t-s)\right\}  Z(t-s)$. Defining $\bar Z(t)= \exp\left\{\bar c^2t\right\}  Z(t)$, we see that this function satisfies \eqref{eqbarZ} using the product rule.
\end{proof}
% Let $\tau=[\tau_1, \tau_2]_{v_0}\in\mathcal T_i$, $i\geq 2$, be a binary tree with root $v_0$. Then, we define
% \begin{align*}
%  \bar F_{\bullet}(t_\ell) &= \Phi(t_\ell,  0) B\\
% \bar F_{\tau}(t_{v_0}, (t_{v_1})_{v_1\in\tau_1}, (t_{v_2})_{v_2\in\tau_2}) &= \Phi(t_{v_0}, 0) H\big(\bar F_{\tau_1}((t_{v_1})_{v_1\in\tau_1})\otimes \bar F_{\tau_2}((t_{v_2})_{v_2\in\tau_2})\big).
%              \end{align*}
\begin{thm}\label{bound_tree_expression} It holds for all $t\in [0, T]$ that
 \begin{align*}
\sum_{\tau\in\mathcal T_i}\int_{D(\tau; t)} F_{\tau}(t, (t_v)_{v\in\tau})F_{\tau}(t, (t_v)_{v\in\tau})^\top \dd(t_v)_{v\in\tau}\leq  \int_0^T Z_{i, T}(s)\dd s \Big(\exp\Big\{\int_0^T \left\|\mathfrak u(s)\right\|^2 \dd s\Big\}\Big)^{2i-1},\end{align*}
% If the control $\mathfrak u$ is additionally bounded by $\bar c>0$, we have \begin{align*}
%  \sum_{\tau\in\mathcal T_i}\int_{D(\tau; t)} F_{\tau}(t, (t_v)_{v\in\tau})F_{\tau}(t, (t_v)_{v\in\tau})^\top \dd(t_v)_{v\in\tau}\leq \int_0^T \bar Z_{i, T}(s)\dd s,
%  \end{align*}
 where $Z_{1, T}(t)$, $t\geq 0$, satisfies %and $\bar Z_{1, T}$ satisfy
 \begin{align}\label{eq_for_Z1}
 \dot {Z}_{1, T}(t) = \mathcal L\big(Z_{1, T}(t)\big),\quad Z_{1, T}(0) = BB^\top %\quad \text{and}\quad
% \dot {\bar Z}_{1, T}(t) = \bar{\mathcal L}\big(\bar Z_{1, T}(t)\Big) ,\quad \bar Z_{1, T}(0) = BB^\top
\end{align}
with $\mathcal L$ %and $\bar{\mathcal L}$
defined in \eqref{lyap_op}.
Further, $Z_{i, T}(t)$, $t\geq 0$,  %and $\bar Z_{i, T}$,
($i\geq 2$) are the solutions of \begin{align}\label{eq_for_Z2}
 \dot {Z}_{i, T}(t) = \mathcal L\big(Z_{i, T}(t)\big),\quad Z_{i, T}(0) = H\big(\sum_{k=1}^{i-1}\int_0^T Z_{k, T}(s)\dd s\otimes \int_0^T Z_{i-k, T}(s) \dd s\big)H^\top.
%  \dot {\bar Z}_{i, T}(t) &= \bar{\mathcal L}\big(\bar Z_{i, T}(t)\Big),\quad \bar Z_{i, T}(0) = H\big(\sum_{k=1}^{i-1}\int_0^T \bar Z_{k, T}(s)\dd s\otimes \int_0^T \bar Z_{i-k, T}(s) \dd s\big)H^\top.
\end{align}
\end{thm}
\begin{proof}
Let us prove the result for $\mathcal T_1=\{\bullet\}$ first. We find that
{\allowdisplaybreaks\begin{align*}
&\sum_{\tau\in\mathcal T_1}\int_{D(\tau; t)} F_{\tau}(t, (t_v)_{v\in\tau})F_{\tau}(t, (t_v)_{v\in\tau})^\top \dd(t_v)_{v\in\tau} = \int_0^t\Phi(t, t_\ell) BB^\top \Phi(t, t_\ell)^\top \dd t_\ell \\
&\leq \int_0^t  \exp\left\{\int_{t_\ell}^t \left\|\mathfrak u(v)\right\|^2 \mathrm{d}v\right\}  Z_{1, T}(t-t_\ell)\dd t_\ell \leq  \exp\left\{\int_{0}^t \left\|\mathfrak u(v)\right\|^2 \mathrm{d}v\right\} \int_0^t  Z_{1, T}(t_\ell)\dd t_\ell
\end{align*}
using Lemma \ref{fund_est} and variable substitution.
% Assuming a bounded control, we obtain \begin{align*}
% &\sum_{\tau\in\mathcal T_1}\int_{D(\tau; t)} F_{\tau}(t, (t_v)_{v\in\tau})F_{\tau}(t, (t_v)_{v\in\tau})^\top \dd(t_v)_{v\in\tau} = \int_0^t\Phi(t, t_\ell) BB^\top \Phi(t, t_\ell)^\top \dd t_\ell \\
% &\leq \int_0^t  \bar Z_{1, T}(t-t_\ell)\dd t_\ell \leq \int_0^t  \bar Z_{1, T}(t_\ell)\dd t_\ell.
% \end{align*}
Now, let the statement of this theorem be true for all sets of binary trees $\mathcal T_k$ with $k=1, \dots, i-1$. For the induction step, let $\tau=[\tau_1, \tau_2]_{v_0}\in\mathcal T_i$ be a tree with root $v_0$. Then, \begin{align*}
& \sum_{\tau\in\mathcal T_i}\int_{D(\tau; t)} F_{\tau}(t, (t_v)_{v\in\tau})F_{\tau}(t, (t_v)_{v\in\tau})^\top \dd(t_v)_{v\in\tau} =   \sum_{k=1}^{i-1}\sum_{\tau_1\in\mathcal T_k}\sum_{\tau_2\in\mathcal T_{i-k}} \int_0^t \int_{D(\tau_2; t_{v_0})} \int_{D(\tau_1; t_{v_0})}\\
 &    F_{\tau}(t,  t_{v_0}, (t_{v_1})_{v_1\in\tau_1}, (t_{v_2})_{v_2\in\tau_2})F_{\tau}(t, t_{v_0}, (t_{v_1})_{v_1\in\tau_1}, (t_{v_2})_{v_2\in\tau_2})^\top \dd(t_{v_1})_{v_1\in\tau_1} \dd(t_{v_2})_{v_2\in\tau_2} \dd t_{v_0}\\
 &=
 \sum_{k=1}^{i-1}\sum_{\tau_1\in\mathcal T_k}\sum_{\tau_2\in\mathcal T_{i-k}} \int_0^t \int_{D(\tau_2; t_{v_0})} \int_{D(\tau_1; t_{v_0})} \Phi(t, t_{v_0}) H\Big(F_{\tau_1}(t_{v_0}, (t_{v_1})_{v_1\in\tau_1} F_{\tau_1}(t_{v_0}, (t_{v_1})_{v_1\in\tau_1})^\top\\
 &    \quad\otimes F_{\tau_2}(t_{v_0}, (t_{v_2})_{v_2\in\tau_2}) F_{\tau_2}(t_{v_0}, (t_{v_2})_{v_2\in\tau_2})^\top\Big)H^\top  \Phi(t, t_{v_0})^\top\dd(t_{v_1})_{v_1\in\tau_1} \dd(t_{v_2})_{v_2\in\tau_2} \dd t_{v_0} \\
 &=
 \sum_{k=1}^{i-1} \int_0^t  \Phi(t, t_{v_0}) H\Big(\sum_{\tau_1\in\mathcal T_k} \int_{D(\tau_1; t_{v_0})}  F_{\tau_1}(t_{v_0}, (t_{v_1})_{v_1\in\tau_1} F_{\tau_1}(t_{v_0}, (t_{v_1})_{v_1\in\tau_1})^\top \dd(t_{v_1})_{v_1\in\tau_1}\\
 &    \quad\otimes \sum_{\tau_2\in\mathcal T_{i-k}} \int_{D(\tau_2; t_{v_0})} F_{\tau_2}(t_{v_0}, (t_{v_2})_{v_2\in\tau_2}) F_{\tau_2}(t_{v_0}, (t_{v_2})_{v_2\in\tau_2})^\top \dd(t_{v_2})_{v_2\in\tau_2}  \Big)H^\top  \Phi(t, t_{v_0})^\top \dd t_{v_0}.
 \end{align*}
 We apply the induction hypothesis together with the monotonicity of the Kronecker product. Specifically, if
$M_1,\bar M_1,M_2,\bar M_2$
are positive semidefinite and
$M_1\leq \bar M_1$, $M_2\leq \bar M_2$,
then
\begin{align*}
M_1\otimes M_2\leq \bar M_1\otimes M_2,
\qquad
M_1\otimes M_2\leq M_1\otimes \bar M_2.
\end{align*}
This yields \begin{align*}
& \sum_{\tau\in\mathcal T_i}\int_{D(\tau; t)} F_{\tau}(t, (t_v)_{v\in\tau})F_{\tau}(t, (t_v)_{v\in\tau})^\top \dd(t_v)_{v\in\tau} \\
 &\leq
 \sum_{k=1}^{i-1} \int_0^t  \Phi(t, t_{v_0}) H\Big(\int_0^T Z_{k, T}(s)\dd s \Big(\exp\Big\{\int_0^T \left\|\mathfrak u(s)\right\|^2 \dd s\Big\}\Big)^{2k-1}\\
 &    \quad\otimes \int_0^T Z_{i-k, T}(s)\dd s \Big(\exp\Big\{\int_0^T \left\|\mathfrak u(s)\right\|^2 \dd s\Big\}\Big)^{2(i-k)-1} \Big)H^\top  \Phi(t, t_{v_0})^\top \dd t_{v_0}\\
 &=
 \Big(\exp\Big\{\int_0^T \left\|\mathfrak u(s)\right\|^2 \dd s\Big\}\Big)^{2i-2}\int_0^t  \Phi(t, t_{v_0}) H\Big(\sum_{k=1}^{i-1} \int_0^T Z_{k, T}(s)\dd s
 \otimes \int_0^T Z_{i-k, T}(s)\dd s\Big) H^\top  \Phi(t, t_{v_0})^\top \dd t_{v_0}\\
 &\leq \Big(\exp\Big\{\int_0^T \left\|\mathfrak u(s)\right\|^2 \dd s\Big\}\Big)^{2i-2}\int_0^t  \exp\left\{\int_{t_{v_0}}^t \left\|\mathfrak u(s)\right\|^2 \mathrm{d}s\right\}  Z_{i, T}(t-t_{v_0})\dd t_{v_0}\\
 &\leq \Big(\exp\Big\{\int_0^T \left\|\mathfrak u(s)\right\|^2 \dd s\Big\}\Big)^{2i-1}\int_0^T  Z_{i, T}(t_{v_0})\dd t_{v_0}
\end{align*}
exploiting Lemma \ref{fund_est},  variable substitution and the definition of $Z_{i, T}$.
% Further, if $\mathfrak u$ is bounded by $\bar c>0$, we obtain  \begin{align*}
% & \sum_{\tau\in\mathcal T_i}\int_{D(\tau; t)} F_{\tau}(t, (t_v)_{v\in\tau})F_{\tau}(t, (t_v)_{v\in\tau})^\top \dd(t_v)_{v\in\tau} \\
%  &\leq
%  \sum_{k=1}^{i-1} \int_0^t  \Phi(t, t_{v_0}) H\Big(\int_0^T \bar Z_{k, T}(s)\dd s \otimes \int_0^T \bar Z_{i-k, T}(s)\dd s  \Big)H^\top  \Phi(t, t_{v_0})^\top \dd t_{v_0}\\
%  &=
%  \int_0^t  \Phi(t, t_{v_0}) H\Big(\sum_{k=1}^{i-1} \int_0^T \bar Z_{k, T}(s)\dd s
%  \otimes \int_0^T \bar Z_{i-k, T}(s)\dd s\Big) H^\top  \Phi(t, t_{v_0})^\top \dd t_{v_0}\\
%  &\leq \int_0^t  \bar Z_{i, T}(t-t_{v_0})\dd t_{v_0} \leq \int_0^T  \bar Z_{i, T}(t_{v_0})\dd t_{v_0}
% \end{align*}
% using Lemma \ref{fund_est},  variable substitution and the definition of $\bar Z_i$.
}
\end{proof}
Below, let us illustrate that $\int_0^T Z_{i, T}(s)\dd s$ %and $\int_0^T \bar Z_{i, T}(s)\dd s$
in Theorem \ref{bound_tree_expression} can be replaced by their limits as $T\to\infty$. This limit is denoted by $\int_0^\infty Z_{i, \infty}(s)\dd s$, where $Z_{i, \infty}$ solves equations \eqref{eq_for_Z1} and \eqref{eq_for_Z2} when $T$ is replaced by $\infty$ in the initial states.
\begin{prop}\label{prop_lim_upper_bound}
 Suppose that $\int_0^\infty Z_{i, \infty}(s)\dd s$ %and $\int_0^\infty \bar Z_{i, \infty}(s)\dd s$
 exist. Then, we have $\int_0^T Z_{i, T}(s)\dd s\leq \int_0^\infty Z_{i, \infty}(s)\dd s$.  %and $\int_0^T \bar Z_{i, T}(s)\dd s\leq\int_0^\infty \bar Z_{i, \infty}(s)\dd s$.
 Moreover, it holds that \begin{align*}
\lim_{T\to \infty}\int_0^T Z_{i, T}(s)\dd s=\int_0^\infty Z_{i, \infty}(s)\dd s.% \quad    \lim_{T\to \infty}\int_0^T \bar Z_{i, T}(s)\dd s=\int_0^\infty \bar Z_{i, \infty}(s)\dd s.
\end{align*}
\end{prop}
\begin{proof}
 The claim follows for $i=1$, since $Z_{1, T}(s)$ is a positive semidefinite matrix and independent of $T$. Therefore,  we obtain $Z_{1, T}(s)=Z_{1, \infty}(s)$ and hence $\int_0^T Z_{1, T}(s)\dd s\leq \int_0^\infty Z_{1, \infty}(s)\dd s$. Let us now assume that $\int_0^T Z_{k, T}(s)\dd s\leq \int_0^\infty Z_{k, \infty}(s)\dd s$ is true for all $k=1, \dots, i-1$. We can represent $Z$ based on the stochastic fundamental solution defined in \eqref{stoch_fund}. Thus,
 \begin{align*}
 Z_{i, T}(s)&= \mathbb E[\Phi_w(s) Z_{i, T}(0)\Phi_w(s)^\top]=   \mathbb E[\Phi_w(s)H\big(\sum_{k=1}^{i-1}\int_0^T Z_{k, T}(s)\dd s\otimes \int_0^T Z_{i-k, T}(s) \dd s\big)H^\top\Phi_w(s)^\top]\\
 &\leq \mathbb E[\Phi_w(s)H\big(\sum_{k=1}^{i-1}\int_0^\infty Z_{k, \infty}(s)\dd s\otimes \int_0^\infty Z_{i-k, \infty}(s) \dd s\big)H^\top\Phi_w(s)^\top]= Z_{i, \infty}(s).
 \end{align*}
Consequently, we obtain $\int_0^T Z_{i, T}(s)\dd s\leq \int_0^\infty Z_{i, \infty}(s)\dd s$. Moreover, we see that $\lim_{T\to\infty}\int_0^T Z_{1, T}(s)\dd s = \int_0^\infty Z_{1, \infty}(s)\dd s$. Now, we assume that $\lim_{T\to\infty}\int_0^T Z_{k, T}(s)\dd s = \int_0^\infty Z_{k, \infty}(s)\dd s$ holds for all $k=1, \dots, i-1$. We exploit that $Z_{i, T}(s)=\expn^{\mathcal L s}Z_{i, T}(0)$ and $Z_{i, \infty}(s)=\expn^{\mathcal L s}Z_{i, \infty}(0)$ in the induction step. Consequently, we find that \begin{align*}
 \lim_{T\to \infty}\int_0^T Z_{i, T}(s)\dd s &= \lim_{T\to \infty}\int_0^T \expn^{\mathcal L s} \dd s \, H\big(\sum_{k=1}^{i-1}\int_0^T Z_{k, T}(s)\dd s\otimes \int_0^T Z_{i-k, T}(s) \dd s  \big)H^\top\\
 &=\int_0^\infty \expn^{\mathcal L s} \dd s\,  H\big(\sum_{k=1}^{i-1}\int_0^\infty Z_{k, \infty}(s)\dd s\otimes \int_0^\infty Z_{i-k, \infty}(s) \dd s\big)H^\top\\
 &=\int_0^\infty \expn^{\mathcal L s} \dd s\, Z_{i, \infty}(0)=\int_0^\infty Z_{i, \infty}(s) \dd s.
 \end{align*}
 This concludes the proof.
 % The proof for $\bar Z$ follows the same steps and is therefore omitted.
\end{proof}
In the following, let us briefly discuss the existence of $\int_0^\infty Z_{i, \infty}(s)\dd s$. %and $\int_0^\infty \bar Z_{i, \infty}(s)\dd s$, respectively.
\begin{prop}\label{int_Zi_exists}
 Let the solution $Z$ of \eqref{eqZ} decay exponentially for $Z(0)= B B^\top$ and $Z(0)= H H^\top$. Then, the integrals $\int_0^\infty Z_{i, \infty}(s)\dd s$ exist.
%  If the solution $\bar Z$ of \eqref{eqbarZ} decays exponentially for $\bar Z(0)= B B^\top$ and $\bar Z(0)= H H^\top$, the integrals $\int_0^\infty \bar Z_{i, \infty}(s)\dd s$ exist.
\end{prop}
\begin{proof}
By assumption, $Z_{1, \infty}$ decays exponentially. Hence, $\int_0^\infty Z_{1, \infty}(s)\dd s$ exists. Now, let all integrals $\int_0^\infty Z_{k, \infty}(s)\dd s$ exist for $k\in\{1, \dots, i-1\}$. Then, we obtain that $Z_{i, \infty}(0)=H\big(\sum_{k=1}^{i-1}\int_0^\infty Z_{k, \infty}(s)\dd s\otimes \int_0^\infty Z_{i-k, \infty}(s) \dd s\big)H^\top$ exists. Let $\lambda_{\max}\geq 0$ be the largest eigenvalue of $\sum_{k=1}^{i-1}\int_0^\infty Z_{k, \infty}(s)\dd s\otimes \int_0^\infty Z_{i-k, \infty}(s) \dd s$. Consequently, we find that \begin{align*}
 Z_{i, \infty}(s)&= \mathbb E[\Phi_w(s) Z_{i, \infty}(0)\Phi_w(s)^\top] \leq \lambda_{\max}\mathbb E[\Phi_w(s) H H^\top\Phi_w(s)^\top].
 \end{align*}
$Z=\mathbb E[\Phi_w H H^\top\Phi_w^\top]$ is the solution of \eqref{eqZ} with $Z(0)=H H^\top$. This function decays exponentially by assumption. Therefore, $\int_0^\infty Z_{i, \infty}(s)\dd s$ is well-defined.
%We omit the proof for $\bar Z$, since it is using exactly the same arguments.
\end{proof}
\begin{remark}\label{remark_stability}
We meet the assumptions of Proposition \ref{int_Zi_exists}, if  \eqref{eqZ} %and \eqref{eqbarZ}
is asymptotically stable, i.e.,  $\sigma({\mathcal L})\subset \mathbb C_-$, %(or $\sigma(\bar{\mathcal L})\subset \mathbb C_-$)
 where $\sigma(\cdot)$ denotes the spectrum of an operator or a matrix.  Notice that this stability condition is equivalent to \begin{align}\label{stab_cond}
 \sigma\big(A\otimes I + I \otimes A +\sum_{k=1}^m N_k\otimes N_k\big)\subset \mathbb C_-                                                                                                                                                                                                                                                                                                                                                                                                                                                                                                                                                                                                                                                                                                                                                                                                                                                                                                                                                                                                                                                                                                                                                                                                                                                                                                                                                                                                                                                                                                                                                                                                                                                                                                                                                                                                                                                                                                                                      \end{align}
with $A\otimes I + I \otimes A +\sum_{k=1}^m N_k\otimes N_k$ being the matrix representation of $\mathcal L$. However, Proposition \ref{int_Zi_exists} tells us that we actually need less than \eqref{stab_cond}. For simplicity, we will always work with \eqref{stab_cond} in the following. Notice that \eqref{stab_cond} is equivalent to the exponential decay of $\mathbb E \|\Phi_w(t)\|^2$, $t\geq 0$, where the stochastic fundamental solution $\Phi_w$ is defined in \eqref{stoch_fund}. We refer to \cite{damm} for more details.
\end{remark}
We introduce reachability Gramians below and use the above result to link this Gramian to dominant subspaces of the quadratic equation \eqref{stochstatenew}. Notice that our approach involves solutions of matrix differential equations while \cite{morBenG24, nonlinear_irka} rely on Volterra kernels.
\begin{defn}\label{def_P}
 Let the matrix-valued functions $Z_{i, T}$, %and $\bar Z_{i, T}$,
 $i\geq 1$, be defined as the solutions of \eqref{eq_for_Z1} and \eqref{eq_for_Z2}. Moreover, let $Z_{i, \infty}$ %and $\bar Z_{i, \infty}$
 be the functions being the solutions of \eqref{eq_for_Z1} and \eqref{eq_for_Z2} when replacing $T$ by $\infty$ in the initial values. Then, the time-limited reachability Gramian $P_T$ and the (infinite time) reachability Gramian $P$ are defined by \begin{align*}
 P_T:=\sum_{i=1}^\infty\int_0^T Z_{i, T}(s)\dd s,\quad     P:=\sum_{i=1}^\infty\int_0^\infty Z_{i, \infty}(s)\dd s.                                                                                                                                                                                                                                                                                                                                                                                                                                                                           \end{align*}
\end{defn}
% \begin{remark}\label{remark_alternative_gram}
% If the control $\mathfrak{u}$ is additionally bounded by a constant $\bar c>0$, we can use alternative Gramians $\bar P_T$ and $\bar P$ defined by \begin{align*}
%  \bar P_T:=\sum_{i=1}^\infty\int_0^T \bar Z_{i, T}(s)\dd s,\quad     \bar P:=\sum_{i=1}^\infty\int_0^\infty \bar Z_{i, \infty}(s)\dd s.                                                                                                                                                                                                                                                                                                                                                                                                                                                                           \end{align*}
% All following results can also be established for $\bar P_T$ and $\bar P$, respectively. The gain of these Gramians is that the exponential $\exp\Big\{\int_0^T \left\|\mathfrak u(s)\right\|^2 \dd s\Big\}$ can be replace by $1$ in all the estimates, at the cost of having to deal with a shifted Lyapunov operator $\bar{\mathcal L}$ involving $A_{\bar c}:=A+I \bar c^2/2$ instead of $A$. This is a consequence of the second matrix inequality in Theorem \ref{bound_tree_expression}.  We omit all further details in this direction.
% \end{remark}
Now, we obtain the following result which is an output bound based on the Gramians introduced in Definition \ref{def_P}.
\begin{thm}\label{thm_gramian_nased bound}
Given the sequence of bilinear systems in \eqref{eq_x1} and \eqref{eq_xi}. Then, we have
\begin{equation}\label{Gramian_based_bound}
 \begin{aligned}
&\sup_{t\in[0, T]}\left\|C \sum_{i=1}^\infty x_i(t)\right\|\\
&\leq  \sqrt{\trace(C P_T C^\top)}\sqrt{\sum_{i=1}^\infty \Big(\exp\Big\{\int_0^T \left\|\mathfrak u(s)\right\|^2 \dd s\Big\}\Big)^{2i-1}\sum_{\tau\in\mathcal T_i}\int_{D(\tau; T)}\prod_{\ell\in\leaves(\tau)}\|u(t_\ell)\|^2\dd(t_v)_{v\in\tau}}.
\end{aligned}
\end{equation}
The inequality in \eqref{Gramian_based_bound} also holds if we replace $P_T$ by $P$.
% If $\mathfrak{u}$ is further bounded,  we obtain \begin{equation}\label{Gramian_based_bound2}
%  \begin{aligned}
% \sup_{t\in[0, T]}\left\|C \sum_{i=1}^\infty x_i(t)\right\|
% \leq  \sqrt{\trace(C \bar P_T C^\top)}\sqrt{\sum_{i=1}^\infty \sum_{\tau\in\mathcal T_i}\int_{D(\tau; T)}\prod_{\ell\in\leaves(\tau)}\|u(t_\ell)\|^2\dd(t_v)_{v\in\tau}}
% \end{aligned}
% \end{equation}
% with $\bar P_T$ introduced in Remark \ref{remark_alternative_gram}. Again, $\bar P_T$ can be replaced by $\bar P$ in \eqref{Gramian_based_bound2}.
\end{thm}
\begin{proof}
{\allowdisplaybreaks
Applying Corollary \ref{kor_first_bound}, \eqref{rel_via_Frob} and Theorem \ref{bound_tree_expression}, we obtain
\begin{align*}
&\left\|C \sum_{i=1}^\infty x_i(t)\right\|\leq
\sum _{i=1}^\infty\left\|C x_i(t)\right\|\\
&\leq \sum _{i=1}^\infty\sqrt{\sum_{\tau\in\mathcal T_i}\int_{D(\tau; t)} \left\|C F_{\tau}(t, (t_v)_{v\in\tau})\right\|_F^2 \dd(t_v)_{v\in\tau}} \sqrt{\sum_{\tau\in\mathcal T_i}\int_{D(\tau; t)}\prod_{\ell\in\leaves(\tau)}\|u(t_\ell)\|^2\dd(t_v)_{v\in\tau}}\\
&\leq \sum _{i=1}^\infty\sqrt{\trace(C \int_0^T Z_{i, T}(s)\dd s\, C^\top)} \sqrt{\Big(\exp\Big\{\int_0^T \left\|\mathfrak u(s)\right\|^2 \dd s\Big\}\Big)^{2i-1}\sum_{\tau\in\mathcal T_i}\int_{D(\tau; T)}\prod_{\ell\in\leaves(\tau)}\|u(t_\ell)\|^2\dd(t_v)_{v\in\tau}}
\end{align*}
exploiting that $D(\tau; t)\subseteq D(\tau; T)$ for all $t\in[0, T]$. Using the Cauchy-Schwarz inequality leads to \begin{align*}
&\left\|C \sum_{i=1}^\infty x_i(t)\right\|\\
&\leq \sqrt{\trace(C P_T C^\top)} \sqrt{\sum _{i=1}^\infty\Big(\exp\Big\{\int_0^T \left\|\mathfrak u(s)\right\|^2 \dd s\Big\}\Big)^{2i-1}\sum_{\tau\in\mathcal T_i}\int_{D(\tau; T)}\prod_{\ell\in\leaves(\tau)}\|u(t_\ell)\|^2\dd(t_v)_{v\in\tau}}
\end{align*}
for all $t\in[0, T]$ yielding \eqref{Gramian_based_bound}. Replacing $P_T$ by $P$ is possible due to Proposition \ref{prop_lim_upper_bound} giving us $P_T\leq P$.
% The steps to prove \eqref{Gramian_based_bound2} are exactly the same and are therefore omitted. This concludes the proof.
}
\end{proof}
Note that Theorem \ref{thm_gramian_nased bound} addresses an open question in \cite{morBenG24, nonlinear_irka}. In particular, $\mathcal{H}_2$-optimal model reduction \cite{nonlinear_irka} minimizes system norms that could not be linked to the output of quadratic bilinear systems so far. This connection is now established by Theorem \ref{thm_gramian_nased bound}. It solely remains to show that $\sum_{i=1}^\infty x_i$ is the solution of the bilinear quadratic equation \eqref{stochstatenew}. Consequently, Theorem \ref{thm_gramian_nased bound} leads to an output bound for the state variable $x$.
\begin{kor}\label{x_eq_sum}
Suppose that the Gramian $P_T$ exists and that the controls $u, \mathfrak u\in L^2_T$ satisfy \begin{align}\label{control_cond1}
\sum _{i=1}^\infty\Big(\exp\Big\{\int_0^T \left\|\mathfrak u(s)\right\|^2 \dd s\Big\}\Big)^{2i-1}\sum_{\tau\in\mathcal T_i}\int_{D(\tau; T)}\prod_{\ell\in\leaves(\tau)}\|u(t_\ell)\|^2\dd(t_v)_{v\in\tau}<\infty.
\end{align}
% Or, if $\mathfrak u$ is bounded by $\bar c>0$, the associated Gramian $\bar P_T$ exist,  and the control $u$ satisfies \begin{align}\label{control_cond2}
% \sum _{i=1}^\infty\sum_{\tau\in\mathcal T_i}\int_{D(\tau; T)}\prod_{\ell\in\leaves(\tau)}\|u(t_\ell)\|^2\dd(t_v)_{v\in\tau}<\infty.
% \end{align}
Then, $x = \sum_{i=1}^\infty x_i$ is the unique global solution of \eqref{stochstatenew} on $[0, T]$.
\end{kor}
\begin{proof}
Setting $C=I$ in Theorem \ref{thm_gramian_nased bound}, we see that the right-hand side of \eqref{Gramian_based_bound} is finite if $P_T$ exists and \eqref{control_cond1} holds.
%If $\bar P_T$ exists and \eqref{control_cond2} is true, then the right-hand side of \eqref{Gramian_based_bound2} is finite. In both cases,
Now, the assumptions of
Proposition \ref{prop_global_sol} are satisfied and hence the result of this corollary follows.
\end{proof}
Generally, it is hard to see a-priori when condition \eqref{control_cond1} % and \eqref{control_cond2} are
is satisfied. We provide a short result for the case of a bounded control $u$.
\begin{prop}\label{prop_bounds_on_control_part}
Let $u$ be bound by a constant $c$ on $[0, T]$. Then, it holds that \begin{align*}
&\sum _{i=1}^\infty\Big(\exp\Big\{\int_0^T \left\|\mathfrak u(s)\right\|^2 \dd s\Big\}\Big)^{2i-1}\sum_{\tau\in\mathcal T_i}\int_{D(\tau; T)}\prod_{\ell\in\leaves(\tau)}\|u(t_\ell)\|^2\dd(t_v)_{v\in\tau}\\
&\leq \exp\Big\{\int_0^T \left\|\mathfrak u(s)\right\|^2 \dd s\Big\}Tc^2 \sum_{j=0}^\infty \bigg(\exp\Big\{2\int_0^T \left\|\mathfrak u(s)\right\|^2 \dd s\Big\}T^2c^2 0.5\bigg)^j.
\end{align*}
%as well as
%\begin{align*}
%\sum _{i=1}^\infty\sum_{\tau\in\mathcal T_i}\int_{D(\tau; T)}\prod_{\ell\in\leaves(\tau)}\|u(t_\ell)\|^2\dd(t_v)_{v\in\tau}\leq Tc^2 \sum_{j=0}^\infty (T^2c^2 0.5)^j.
%\end{align*}
\end{prop}
\begin{proof}
{\allowdisplaybreaks
Since $u$ is bounded, we find that
\begin{align*}
 \sum_{\tau\in\mathcal T_i}\int_{D(\tau; T)}\prod_{\ell\in\leaves(\tau)}\|u(t_\ell)\|^2\dd(t_v)_{v\in\tau}
 &\leq \sum_{\tau\in\mathcal T_i} \Vol(D(\tau; T)) \, c^{2 \vert \leaves(\tau)\vert}=\sum_{\tau\in\mathcal T_i} \frac{T^{\vert\tau \vert}}{\gamma(\tau)} c^{2 \vert \leaves(\tau)\vert}\\
 &=T^{2i-1} c^{2i}\sum_{\tau\in\mathcal T_i} \frac{1}{\gamma(\tau)}
\end{align*}
using \eqref{vol_D_tau} and the definition of $\mathcal T_i$ in \eqref{def_bin_trees}. We define the sequence $S_i:=\sum_{\tau\in\mathcal T_i} \frac{1}{\gamma(\tau)}$ with $S_1=1$ and observe that \begin{align*}
 S_i=\sum_{k=1}^{i-1}\sum_{\tau_1\in\mathcal T_k}\sum_{\tau_2\in\mathcal T_{i-k}} \frac{1}{\vert\tau\vert\gamma(\tau_1)\gamma(\tau_2)}  = \frac{1}{2i-1} \sum_{k=1}^{i-1} \sum_{\tau_1\in\mathcal T_k}\frac{1}{\gamma(\tau_1)} \sum_{\tau_2\in\mathcal T_{i-k}}\frac{1}{\gamma(\tau_2)}   =  \frac{1}{2i-1} \sum_{k=1}^{i-1} S_k S_{i-k}                                                                                                                                                                                          \end{align*}
for $i\geq 2$ using the definition of the tree factorial of $\tau=[\tau_1, \tau_2]$ and $\vert\tau\vert=2i-1$. We prove that $S_i\leq 0.5^{i-1}$ for all $i\geq 1$. The claim is true for $i=1$. Let us assume that it also holds for $S_1, \dots, S_{i-1}$. Then, we obtain \begin{align*}
 S_i\leq  \frac{1}{2i-1} \sum_{k=1}^{i-1} 0.5^{k-1} 0.5^{i-k-1}=   \frac{i-1}{2i-1} 0.5^{i-2}\leq  0.5^{i-1},                                                                                                                                                                                                                                                                    \end{align*}
as $\frac{i-1}{2i-1}\leq 0.5$ for all $i\geq 1$. Consequently, we have \begin{align*}
 \sum_{\tau\in\mathcal T_i}\int_{D(\tau; T)}\prod_{\ell\in\leaves(\tau)}\|u(t_\ell)\|^2\dd(t_v)_{v\in\tau}
 \leq T^{2i-1} c^{2i} 0.5^{i-1}
 \end{align*}
  and hence
  %\begin{align*}
%  \sum_{i=1}^\infty\sum_{\tau\in\mathcal T_i}\int_{D(\tau; T)}\prod_{\ell\in\leaves(\tau)}\|u(t_\ell)\|^2\dd(t_v)_{v\in\tau}
%  &\leq \sum_{i=1}^\infty T^{2i-1} c^{2i} 0.5^{i-1}= \sum_{j=0}^\infty T^{2j+1} c^{2j+2} 0.5^{j}\\
%  &=Tc^2 \sum_{j=0}^\infty (T^2c^2 0.5)^j.
%  \end{align*}
% With the additional exponential term, the estimate is
\begin{align*}
 &\sum_{i=1}^\infty\Big(\exp\Big\{\int_0^T \left\|\mathfrak u(s)\right\|^2 \dd s\Big\}\Big)^{2i-1}\sum_{\tau\in\mathcal T_i}\int_{D(\tau; T)}\prod_{\ell\in\leaves(\tau)}\|u(t_\ell)\|^2\dd(t_v)_{v\in\tau}\\
 &\leq  \sum_{j=0}^\infty \Big(\exp\Big\{\int_0^T \left\|\mathfrak u(s)\right\|^2 \dd s\Big\}\Big)^{2j+1} T^{2j+1} c^{2j+2} 0.5^{j}\\
 &=\exp\Big\{\int_0^T \left\|\mathfrak u(s)\right\|^2 \dd s\Big\}Tc^2 \sum_{j=0}^\infty \bigg(\exp\Big\{2\int_0^T \left\|\mathfrak u(s)\right\|^2 \dd s\Big\}T^2c^2 0.5\bigg)^j.
 \end{align*}
 }
 \end{proof}
 \begin{remark}
Numerical experiments indicate that the sequence $(S_i)$ defined in the proof of Proposition \ref{prop_bounds_on_control_part} seems to satisfy $S_i\leq q^{i-1}$ for $q=4/\pi^2\approx 0.405285$. This would slightly improve the bounds in Proposition \ref{prop_bounds_on_control_part}, since the constant $0.5$ could be replace by $4/\pi^2$. However, if the conjecture on $S_i$ is true, it cannot be shown via a simple induction proof. From Proposition \ref{prop_bounds_on_control_part} we now conclude that the series in \eqref{control_cond1} converges if \begin{align}\label{cond_con_u}
  \exp\Big\{2\int_0^T \left\|\mathfrak u(s)\right\|^2 \dd s\Big\}T^2c^2 0.5 <1.                                                                                                                                                                                                                                                                                                                                                                                                                                                                                                                                                                                                                                                                                                                                                                                                                                                                                                                                                                                                                                                                          \end{align}
% A criterion for the convergence of the series in  \eqref{control_cond2} is obtained when setting $\mathfrak{u}\equiv 0$ in \eqref{cond_con_u}.
 \end{remark}
Below, let us clarify under  which conditions the reachability Gramians introduced in Definition \ref{def_P} exist. Although solely (conservative) sufficient conditions for their existence are provided, they are more general than the ones given in \cite{morBenG24}.
\begin{prop}\label{prop_reach_gram_exist}
 The Gramian $P_T=\sum_{i=1}^\infty\int_0^T Z_{i, T}(t)\dd t$ exists if \begin{align*}
  4\int_0^T \mathbb E\left\| \Phi_w(t) H\right\|^2\dd t \int_0^T \left\|Z_1(t)\right\|\dd t\leq 1,                                                  \end{align*}
where $Z_1:=Z_{1, T}= Z_{1, \infty}$ solves \eqref{eq_for_Z1} and $\Phi_w$ is defined in \eqref{stoch_fund}.  Given the stability condition \eqref{stab_cond}, the infinite Gramian $P=\sum_{i=1}^\infty\int_0^\infty Z_{i, \infty}(t)\dd t$ exists if \begin{align*}
  4\int_0^\infty \mathbb E\left\| \Phi_w(t) H\right\|^2\dd t \int_0^\infty \left\|Z_1(t)\right\|\dd t\leq 1.
                      \end{align*}
\end{prop}
\begin{proof}
We define a sequence $(b_i)$ by  $b_i:=\int_0^T \left\|Z_{i, T}(t)\right\|\dd t$ for $i\geq 1$. We further exploit that $Z_{i, T}(t)=\mathbb E[\Phi_w(t)Z_{i, T}(0)\Phi_w(t)^\top]$, where $Z_{1, T}(0)=B B^\top$ and $Z_{i, T}(0)=H\Big(\sum_{k=1}^{i-1}\int_0^T Z_{k, T}(s)\dd s\otimes \int_0^T Z_{i-k, T}(s) \dd s\Big)H^\top $ for $i\geq 2$.  Consequently, we have \begin{align}\nonumber
    b_i& = \int_0^T \left\|\mathbb E\Big[\Phi_w(t)H\Big(\sum_{k=1}^{i-1}\int_0^T Z_{k, T}(s)\dd s\otimes \int_0^T Z_{i-k, T}(s) \dd s\Big)H^\top \Phi_w(t)^\top\Big]\right\|\dd t \\
    &\leq \int_0^T \mathbb E\left\| \Phi_w(t) H\right\|^2 \dd t \sum_{k=1}^{i-1} \left\|\int_0^T Z_{k, T}(s)\dd s \right\|\left\|\int_0^T Z_{i-k, T}(s) \dd s\right\|\leq k_H \sum_{k=1}^{i-1} b_k b_{i-k}, \label{catalan_inequality2}
                                     \end{align}
with $k_H:=\int_0^T \mathbb E\left\| \Phi_w(t) H\right\|^2\dd t$ for $i\geq 2$. Therefore, $(b_i)$ satisfies a Catalan number type recursion, however, with an inequality. We are now constructing a sequence $(d_i)$ that fulfills \eqref{catalan_inequality2} with an equality using the Catalan numbers $(c_i)$ defined through \eqref{catalan_number_recursion}. We set $d_i = k_H^{i-1} b_1^i c_i$ and observe that $d_1 = b_1$. For $i\geq 2$, we obtain that \begin{align*}
  d_i = k_H^{i-1} b_1^i c_i =k_H^{i-1} b_1^i \sum_{k=1}^{i-1} c_k c_{i-k} = k_H \sum_{k=1}^{i-1} (k_H^{k-1} b_1^k c_k) (k_H^{i-k-1}b_1^{i-k} c_{i-k})  =      k_H \sum_{k=1}^{i-1} d_k d_{i-k}.                                                                                                                                                                                                                                                                                                                                                                                            \end{align*}
We now prove that $b_i\leq d_i$ for all $i\geq 1$. The statement is true for $i=1$, so that we are now assuming that $b_k\leq d_k$ for all $1\leq k\leq i-1$. Consequently, we have \begin{align*}
 b_i\leq  k_H \sum_{k=1}^{i-1} b_k b_{i-k}\leq  k_H \sum_{k=1}^{i-1} d_k d_{i-k}= d_i.                                                                                                                                                                                                                                                                                                                                                                    \end{align*}
As the asymptotics of the Catalan numbers are given by $c_i\sim \frac{4^{i-1}}{(i-1)^{\frac{3}{2}}\sqrt{\pi}}$, we obtain \begin{align*}
  d_i\sim b_1\frac{(4k_H b_1)^{i-1}}{(i-1)^{\frac{3}{2}}\sqrt{\pi}}.                                                                                                                          \end{align*}
Therefore, the right-hand side of \begin{align*}
  \sum_{i=1}^\infty \left\|\int_0^T Z_{i, T}(t)\dd t \right\|\leq \sum_{i=1}^\infty \int_0^T \left\|Z_{i, T}(t)\right\|\dd t =  \sum_{i=1}^\infty b_i\leq \sum_{i=1}^\infty d_i
                                  \end{align*}
is finite if $4k_H b_1\leq 1$. Given \eqref{stab_cond}, the integrals \begin{align*}
b_{1, \infty}:=\int_0^\infty \left\|Z_1(t)\right\|\dd t,\quad k_{H, \infty}:=\int_0^\infty\mathbb E\left\| \Phi_w(t) H\right\|^2\dd t
\end{align*}
 are finite, see Remark \ref{remark_stability}. Replacing $T$ by $\infty$ in the above proof leads to \begin{align*}
 \int_0^\infty \left\|Z_{i, \infty}(t)\right\|\dd t \leq  k_{H, \infty}^{i-1} b_{1, \infty}^i c_i
\end{align*}
for all $i\geq 1$. Using the above argument once more, we know that $\sum_{i=1}^\infty\int_0^\infty \left\|Z_{i, \infty}(t)\right\|\dd t$ is finite if $4k_{H, \infty} b_{1, \infty}\leq 1$. This concludes the proof.
\end{proof}
%\begin{remark}
% Notice that $Z_1$
%\end{remark}
We finally relate the reachability Gramians to matrix equations. These equations play an important role in the error analysis of the reduction procedure developed later, see Sections \ref{sec_exact_rom} and \ref{BT_bound}.
\begin{prop}
 If the Gramian $P_T=\sum_{i=1}^\infty\int_0^T Z_{i, T}(s)\dd s$ exists, then it satisfies \begin{align}\label{eq_for_PT}
   \sum_{i=1}^\infty Z_{i, T}(T)- B B^\top =\mathcal L\big(P_T\big)+ H\big(P_T\otimes P_T\big)H^\top.                                                \end{align}
  If \eqref{stab_cond} holds and $P=\sum_{i=1}^\infty\int_0^\infty Z_{i, \infty}(s)\dd s$ is well-defined, then it is a solution of     \begin{align}\label{eq_for_P}
   - B B^\top =\mathcal L\big(P\big)+ H\big(P\otimes P\big)H^\top.
              \end{align}
\end{prop}
\begin{proof}
 Let us integrate the equations for $Z_{i, T}$ in \eqref{eq_for_Z1} and \eqref{eq_for_Z2}. This yields \begin{align*}
 Z_{i, T}(T)-Z_{i, T}(0) = \mathcal L\Big(\int_0^T Z_{i, T}(s)\dd s\Big)                                                                                                  \end{align*}
using that $\mathcal L$ is linear (and bounded). We sum over $i$ from $1$ to some $\mathfrak n$, use $Z_{1, T}(0)=B B^\top$ and $Z_{i, T}(0)=H\Big(\sum_{j_i, j_2\geq 1 \atop j_1+j_2=i}\int_0^T Z_{j_1, T}(s)\dd s\otimes \int_0^T Z_{j_2, T}(s) \dd s\Big)H^\top$ for $i\geq 2$. This leads to  \begin{align}\label{basic_result}
 \sum_{i=1}^{\mathfrak n} Z_{i, T}(T) - B B^\top = \mathcal L\Big(\sum_{i=1}^{\mathfrak n}\int_0^T Z_{i, T}(s)\dd s\Big) + H\Big(\sum_{j_i, j_2\geq 1 \atop j_1+j_2\leq \mathfrak n}\int_0^T Z_{j_1, T}(s)\dd s\otimes \int_0^T Z_{j_2, T}(s) \dd s\Big)H^\top.                                                                                             \end{align}
 We take the limit as $\mathfrak n\to\infty$ in \eqref{basic_result} and obtain \eqref{eq_for_PT}. Now, taking the limit as $T\to\infty$ in \eqref{basic_result} and using Proposition \ref{prop_lim_upper_bound} yields \begin{align}\label{basic_result2}
 - B B^\top = \mathcal L\Big(\sum_{i=1}^{\mathfrak n}\int_0^\infty Z_{i, \infty}(s)\dd s\Big) + H\Big(\sum_{j_i, j_2\geq 1 \atop j_1+j_2\leq \mathfrak n}\int_0^\infty Z_{j_1, \infty}(s)\dd s\otimes \int_0^\infty Z_{j_2, \infty}(s) \dd s\Big)H^\top.                                                                                             \end{align}
Above, we exploited that \eqref{stab_cond} implies $Z_{i, T}(T)\to 0$ for all $i=1, 2, \dots $ as $T\to \infty$. This can be argued based on the proof of Proposition \ref{prop_lim_upper_bound} giving us $Z_{i, T}(T)\leq Z_{i, \infty}(T)$. Moreover, $Z_{i, \infty}(T)\to 0$ as $T\to \infty$ considering the proof of Proposition \ref{int_Zi_exists}. Taking the limit as $\mathfrak n\to \infty$ in \eqref{basic_result2} concludes the proof.
\end{proof}
It is worth noting that our Gramian approach leads to the same matrix equations as in \cite{morBenG24}.

\section{Dominant subspaces and reduced-order model}

Identifying dominant subspaces is a key step in model reduction. Existing works on balanced truncation for bilinear quadratic systems like \cite{morBenG24} did not establish connections between the eigenspaces of the Gramians and dominant state-space directions. Our approach can be used to address this aspect.

\subsection{Observability Gramian}

We introduce the block-wise transposed of $N$ and $H$ by
\begin{align*}
 \mathcal N:=\begin{bmatrix}N_1^\top & \ldots & N_m^\top\end{bmatrix},\quad  \mathcal H:=\begin{bmatrix}H_1^\top & \ldots & H_n^\top\end{bmatrix}
\end{align*}
and define the adjoint operator w.r.t. $\langle \cdot, \cdot\rangle_F$ of $\mathcal L$ (introduced in \eqref{lyap_op}) as follows
\begin{align}\label{lyap_op_adjoint}
\mathcal L^*(X):= A^\top X + X A + \mathcal N (I\otimes X) \mathcal N^\top.
               \end{align}
We introduce the notion of a (time-limited) observability Gramian. Once more, solutions of matrix differential equations are involved in contrast to the Volterra kernel approach used in \cite{morBenG24, nonlinear_irka}.
\begin{defn}\label{def_Q}
Let $\mathcal Z_{1, T}(t)$, $t\geq 0$, satisfy
 \begin{align}\label{eq_for_Z1_ad}
 \dot {\mathcal Z}_{1, T}(t) = \mathcal L^*\big(\mathcal Z_{1, T}(t)\big),\quad \mathcal Z_{1, T}(0) = C^\top C
\end{align}
with $\mathcal L^*$ defined in \eqref{lyap_op_adjoint}.
Further, $\mathcal Z_{i, T}(t)$, $t\geq 0$,
($i\geq 2$) are assumed to be the solutions of \begin{align}\label{eq_for_Z2_ad}
 \dot {\mathcal Z}_{i, T}(t) = \mathcal L^*\big(\mathcal Z_{i, T}(t)\big),\quad \mathcal Z_{i, T}(0) = \mathcal H\big(\sum_{k=1}^{i-1}\int_0^T Z_{k, T}(s)\dd s\otimes \int_0^T \mathcal Z_{i-k, T}(s) \dd s\big)\mathcal H^\top,\end{align}
 where $Z_{k, T}$, $i\geq 1$, fulfill \eqref{eq_for_Z1} and \eqref{eq_for_Z2}. Moreover, let us introduce $\mathcal Z_{i, \infty}$, $i\geq 1$, as the solutions of \eqref{eq_for_Z1_ad} and \eqref{eq_for_Z2_ad} when $T$ is replaced by $\infty$ in the initial states. Then, we define \begin{align*}                                                                                                                                                                                                                                                                               Q_T:=\sum_{i=1}^\infty\int_0^T \mathcal Z_{i, T}(s)\dd s\quad \text{and}  \quad  Q:=\sum_{i=1}^\infty\int_0^\infty \mathcal Z_{i, \infty}(s)\dd s.                                                                                                                                                                                                                                                                             \end{align*}
\end{defn}
\begin{remark}\label{rel_mathcalZs}
Assume that $\int_0^\infty Z_{i, \infty}(s)\dd s$ and $\int_0^\infty \mathcal Z_{i, \infty}(s)\dd s$
 exist for all $i\geq 1$. Then, relations between $\mathcal Z_{i, T}$ and $\mathcal Z_{i, \infty}$ can be proved completely analogue to Proposition \ref{prop_lim_upper_bound}. These are $\mathcal Z_{i, T}(s)\leq \mathcal Z_{i, \infty}(s)$ for all $s\geq 0$ and hence
  $\int_0^T \mathcal Z_{i, T}(s)\dd s\leq \int_0^\infty \mathcal Z_{i, \infty}(s)\dd s$ as well as \begin{align*}
\lim_{T\to \infty}\int_0^T \mathcal Z_{i, T}(s)\dd s=\int_0^\infty \mathcal Z_{i, \infty}(s)\dd s.
\end{align*}
\end{remark}
Let us briefly show under which conditions the infinite integrals, mentioned in Remark \ref{rel_mathcalZs}, exist.
\begin{prop}\label{int_mathcal_Zi_exists}
 Let \eqref{stab_cond} hold. Then, the integrals $\int_0^\infty \mathcal Z_{i, \infty}(s)\dd s$ exist for all $i\geq 1$.
\end{prop}
\begin{proof}
Assumption \eqref{stab_cond} is equivalent to \begin{align*}
 \sigma\big(A^\top\otimes I + I \otimes A^\top +\sum_{k=1}^m N_k^\top\otimes N_k^\top\big)\subset \mathbb C_-                                                                                                                                                                                                                                                                                                                                                                                                                                                                                                                                                                                                                                                                                                                                                                                                                                                                                                                                                                                                                                                                                                                                                                                                                                                                                                                                                                                                                                                                                                                                                                                                                                                                                                                                                                                                                                                                                                                                      \end{align*}
meaning that $\sigma({\mathcal L^*})\subset \mathbb C_-$. This provides exponential stability of \eqref{eq_for_Z1_ad} and hence the existence of $\int_0^\infty \mathcal Z_{1, \infty}(s)\dd s$.
Suppose that all integrals $\int_0^\infty \mathcal Z_{k, \infty}(s)\dd s$ exist for $k\in\{1, \dots, i-1\}$. Moreover, $\int_0^\infty Z_{k, \infty}(s)\dd s$ exist for all $k\geq 1$ by Proposition \ref{int_Zi_exists} and Remark \ref{remark_stability}. Consequently, the initial state
 $\mathcal Z_{i, \infty}(0)=\mathcal H\big(\sum_{k=1}^{i-1}\int_0^\infty Z_{k, \infty}(s)\dd s\otimes \int_0^\infty \mathcal Z_{i-k, \infty}(s) \dd s\big)\mathcal H^\top$ is well defined.  Therefore, $\mathcal Z_{i, \infty}$ decays exponentially by assumption and hence its infinite integral exists.
\end{proof}
Now, we need to dual stochastic fundamental solution $\bar \Phi_w$ for the next proposition, where we exploit the stochastic representation of the matrix ODEs occuring in Definition \ref{def_Q}. We introduce $\bar \Phi_w$  as the solution of \begin{align}\label{stoch_fund_ad}
 \bar\Phi_w(t) = I +\int_0^t A^\top \bar \Phi_w(v) \dd v + \int_0^t \mathcal N \big(\dd w(v)\otimes \bar \Phi_w(v)\big).
\end{align}
From Lemma \ref{fund_est}, we obtain that $\mathcal Z_{i, T}(t)=\mathbb E[\bar \Phi_w(t)\mathcal Z_{i, T}(0)\bar \Phi_w(t)^\top]$. We apply this identity below.
\begin{prop}\label{prop_obs_gram_exist}
Suppose that $\sum_{i=1}^\infty \int_0^T \left\|Z_{i, T}(t)\right\|\dd t$ is finite (hence $P_T$ exists) and \begin{align*}
  \int_0^T \mathbb E\left\| \bar \Phi_w(t) \mathcal H\right\|^2\dd t \sum_{i=1}^\infty \int_0^T \left\|Z_{i, T}(t)\right\|\dd t <1.
                                  \end{align*}
 Then, the Gramian $Q_T=\sum_{i=1}^\infty\int_0^T \mathcal Z_{i, T}(t)\dd t$ exists. Given the stability condition \eqref{stab_cond}, the convergence of $\sum_{i=1}^\infty \int_0^\infty \left\|Z_{i, \infty}(t)\right\|\dd t$ (hence $P$ exists) and
 \begin{align*}
  \int_0^\infty \mathbb E\left\| \bar \Phi_w(t) \mathcal H\right\|^2\dd t \sum_{i=1}^\infty \int_0^\infty \left\|Z_{i, \infty}(t)\right\|\dd t <1.
                                  \end{align*}
Then, the infinite Gramian $Q=\sum_{i=1}^\infty\int_0^\infty \mathcal Z_{i, \infty}(t)\dd t$ exists.
\end{prop}
\begin{proof}
We define a sequence $(\bar b_i)$ by  $\bar b_i:=\int_0^T \left\|\mathcal Z_{i, T}(t)\right\|\dd t$ for $i\geq 1$ and additionally use the notation of Proposition \ref{prop_reach_gram_exist}. Based on the representation $\mathcal Z_{i, T}(t)=\mathbb E[\bar \Phi_w(t)\mathcal Z_{i, T}(0)\bar \Phi_w(t)^\top]$, where $\mathcal Z_{1, T}(0)=C^\top C$ and $\mathcal Z_{i, T}(0)=\mathcal H\Big(\sum_{k=1}^{i-1}\int_0^T Z_{k, T}(s)\dd s\otimes \int_0^T \mathcal Z_{i-k, T}(s) \dd s\Big)\mathcal H^\top $ for $i\geq 2$, we can now follow the steps in \eqref{catalan_inequality2} to obtain \begin{align}\label{catalan_inequality3}
    \bar b_i\leq \bar k_{\mathcal H} \sum_{k=1}^{i-1} b_k \bar b_{i-k}
                                     \end{align}
with $\bar k_{\mathcal H}:=\int_0^T \mathbb E\left\| \bar \Phi_w(t) \mathcal H\right\|^2\dd t$ for $i\geq 2$. From \eqref{catalan_inequality3}, we conclude that
\begin{align*}
    \sum_{i=1}^\mathfrak n\bar b_i\leq \bar b_1+\bar k_{\mathcal H} \sum_{i=2}^\mathfrak n\sum_{k=1}^{i-1} b_k \bar b_{i-k}=\bar b_1+\bar k_{\mathcal H} \sum_{k=1}^{\mathfrak n-1} \sum_{j=1}^{\mathfrak n-k}b_k \bar b_{j}\leq \bar b_1+\bar k_{\mathcal H} \sum_{k=1}^{\infty} b_k \sum_{j=1}^{\mathfrak n}\bar b_{j}.
                                     \end{align*}
Rearranging this inequality yields
\begin{align*}
    \sum_{i=1}^\mathfrak n\bar b_i\leq \frac{\bar b_1}{1-\bar k_{\mathcal H} \sum_{k=1}^{\infty} b_k}
                                     \end{align*}
given $1-\bar k_{\mathcal H} \sum_{k=1}^{\infty} b_k>0$. This means that the increasing sequence $\Big(\sum_{i=1}^\mathfrak n\bar b_i\Big)$ is bounded and hence convergent if $\bar k_{\mathcal H} \sum_{k=1}^{\infty} b_k<1$. Replacing $T$ by $\infty$ in the above proof, the result also follows for $Q$. This concludes the proof.
 \end{proof}
We also show that the observability Gramians are associated with certain matrix equations. This is used in Sections \ref{sec_dom_sub}, \ref{sec_exact_rom} and \ref{BT_bound} later.
 \begin{prop}
 If $P_T=\sum_{i=1}^\infty\int_0^T Z_{i, T}(s)\dd s$ and $Q_T=\sum_{i=1}^\infty\int_0^T \mathcal Z_{i, T}(s)\dd s$ exist, then $Q_T$ satisfies \begin{align}\label{eq_for_QT}
   \sum_{i=1}^\infty \mathcal Z_{i, T}(T)- C^\top C =\mathcal L^*\big(Q_T\big)+ \mathcal H\big(P_T\otimes Q_T\big)\mathcal H^\top.
                                                \end{align}
  If \eqref{stab_cond} holds,  $P=\sum_{i=1}^\infty\int_0^\infty  Z_{i, \infty}(s)\dd s$ exists and $Q=\sum_{i=1}^\infty\int_0^\infty \mathcal Z_{i, \infty}(s)\dd s$ is well-defined, then $Q$ is a solution of     \begin{align}\label{eq_for_Q}
   - C^\top C =\mathcal L^*\big(Q\big)+ \mathcal H\big(P\otimes Q\big)\mathcal H^\top.
                        \end{align}
\end{prop}
\begin{proof}
We integrate \eqref{eq_for_Z1_ad} and \eqref{eq_for_Z2_ad} yielding \begin{align*}
 \mathcal Z_{i, T}(T)-\mathcal Z_{i, T}(0) = \mathcal L^*\Big(\int_0^T \mathcal Z_{i, T}(s)\dd s\Big), \quad i\geq 1.                                                                                                  \end{align*}
Summing up the first $\mathfrak n$ of these equations and using $\mathcal Z_{1, T}(0)=C^\top C$ as well as $\mathcal Z_{i, T}(0)=\mathcal H\Big(\sum_{j_i, j_2\geq 1 \atop j_1+j_2=i}\int_0^T Z_{j_1, T}(s)\dd s\otimes \int_0^T \mathcal Z_{j_2, T}(s) \dd s\Big)\mathcal H^\top$ for $i\geq 2$, we obtain   \begin{align}\label{basic_result_ad}
 \sum_{i=1}^{\mathfrak n} \mathcal Z_{i, T}(T) - C^\top C = \mathcal L^*\Big(\sum_{i=1}^{\mathfrak n}\int_0^T \mathcal Z_{i, T}(s)\dd s\Big) + \mathcal H\Big(\sum_{j_i, j_2\geq 1 \atop j_1+j_2\leq \mathfrak n}\int_0^T Z_{j_1, T}(s)\dd s\otimes \int_0^T \mathcal Z_{j_2, T}(s) \dd s\Big)\mathcal H^\top.                                                                                             \end{align}
 We take the limit as $\mathfrak n\to\infty$ in \eqref{basic_result_ad} and obtain \eqref{eq_for_QT}. Taking $T\to\infty$ in \eqref{basic_result_ad} and using Proposition \ref{prop_lim_upper_bound} as well as Remark \ref{rel_mathcalZs} yields \begin{align}\label{basic_result2_ad}
 - C^\top C = \mathcal L^*\Big(\sum_{i=1}^{\mathfrak n}\int_0^\infty \mathcal Z_{i, \infty}(s)\dd s\Big) + \mathcal H\Big(\sum_{j_i, j_2\geq 1 \atop j_1+j_2\leq \mathfrak n}\int_0^\infty Z_{j_1, \infty}(s)\dd s\otimes \int_0^\infty \mathcal Z_{j_2, \infty}(s) \dd s\Big)\mathcal H^\top.                                                                                             \end{align}
This result relies on \eqref{stab_cond} implying $\mathcal Z_{i, T}(T)\to 0$ for all $i\geq 1$ as $T\to \infty$, see  Remark \ref{rel_mathcalZs} and the proof of Proposition \ref{int_mathcal_Zi_exists}. The proof is concluded by taking $\mathfrak{n}\to\infty$ in \eqref{basic_result2_ad}.
\end{proof}
Note that the observability Gramians used in this paper lead to the same matrix equations as in \cite{morBenG24}.

\subsection{Identifying dominant subspaces based on Gramians}\label{sec_dom_sub}

In this section, we identify two subspaces in which the state $x$ and the output $y$, respectively, can be well-approximated. To illustrate this, let us use a general orthonormal basis $(v_k)$ of $\mathbb R^n$, giving the representation \begin{align*}
   x(t)=\sum_{k=1}^n\langle x(t), v_k\rangle\, v_k.                                                                                                                                                                                                                                                                                                                                                                                                                                                          \end{align*}
Let us assume that $P_T$ as well as $P$ exist and that the controls $u, \mathfrak u\in L^2_T$ satisfy \eqref{control_cond1}, i.e.,\begin{align}\label{def_fuu}
f(u, \mathfrak{u}):=\sum _{i=1}^\infty\Big(\exp\Big\{\int_0^T \left\|\mathfrak u(s)\right\|^2 \dd s\Big\}\Big)^{2i-1}\sum_{\tau\in\mathcal T_i}\int_{D(\tau; T)}\prod_{\ell\in\leaves(\tau)}\|u(t_\ell)\|^2\dd(t_v)_{v\in\tau}<\infty.
\end{align}
Then, by Theorem \ref{thm_gramian_nased bound} with $C=v_k^\top$ and Corollary \ref{x_eq_sum}, we have
 \begin{align}\label{dominant_subspaceest}
\sup_{t\in[0, T]}\left\vert\langle x(t), v_k\rangle\right\vert
\leq  \sqrt{v_k^\top P_T v_k}\sqrt{f(u, \mathfrak{u})}\leq  \sqrt{v_k^\top P v_k}\sqrt{f(u, \mathfrak{u})}.
\end{align}
Therefore, the direction $v_k$ is of low relevance for $x$ if $v_k^\top P_T v_k$ or $v_k^\top P v_k$ is small. Since $P$ and $P_T$ are positive semidefinite matrices, we may further assume that $v_k$ is an eigenvector of either $P$ or $P_T$ associated with an eigenvalue $\lambda_k \geq 0$. In this case, $v_k$ is negligible whenever $\lambda_k$ is small. Consequently, the eigenspaces corresponding to small eigenvalues of $P$ or $P_T$ have low relevance for the state variable and can therefore be truncated. Secondly, we consider the difference \begin{align*}
  y(t)-C\sum_{j=1 \atop j\neq k}^n\langle x(t), v_j\rangle\, v_j=
  Cv_k \langle x(t), v_k\rangle = Cv_k v_k^\top x(t).
  \end{align*}
  Exploiting Theorem \ref{thm_gramian_nased bound} with $C$ replaced by $Cv_k v_k^\top$, we obtain
  \begin{align}\nonumber
\sup_{t\in[0, T]}\Big\| y(t)-C\sum_{j=1 \atop j\neq k}^n\langle x(t), v_j\rangle\, v_j\Big\|
&\leq  \sqrt{\trace\big(Cv_k v_k^\top P_T (Cv_k v_k^\top)^\top\big)}\sqrt{f(u, \mathfrak{u})}\\ \label{dom_sub_y_PT}
&= \sqrt{v_k^\top P_T v_k} \sqrt{v_k^\top C^\top C v_k}\sqrt{f(u, \mathfrak{u})}\\ \label{dom_sub_y_P}
&\leq \sqrt{v_k^\top P v_k} \sqrt{v_k^\top C^\top C v_k}\sqrt{f(u, \mathfrak{u})}.
\end{align}
Therefore, removing $v_k$ has little impact on the output if \eqref{dom_sub_y_PT} or \eqref{dom_sub_y_P} are small. If $v_k$ is an eigenvector of $P_T$ or $P$, respectively, then a small eigenvalue $\lambda_k$ ensures that the first factor in either \eqref{dom_sub_y_PT} or \eqref{dom_sub_y_P} is small. However, the output approximation can be further improved by incorporating information from $C$ into the reduced model. In particular, it is beneficial to choose a basis $(v_k)$ such that $v_k^\top C^\top C v_k$ is also small.

Now let $v_k$ be an eigenvector of $Q_T$ or $Q$, respectively, corresponding to a small eigenvalue $\mu_k \geq 0$. We now show that this implies $v_k^\top C^\top C v_k$ is small. To see this, multiply \eqref{eq_for_QT} from the left by $v_k^\top$ and from the right by $v_k$, yielding
\begin{align}\label{dominat_direction_q}
v_k^\top C^\top C v_k
= -2\mu_k v_k^\top A v_k
- v_k^\top \mathcal N (I\otimes Q_T) \mathcal N^\top v_k
- v_k^\top \mathcal H (P_T\otimes Q_T)\mathcal H^\top v_k
+ \sum_{i=1}^\infty v_k^\top \mathcal Z_{i,T}(T)v_k.
\end{align}
Here, we exploit the definition of $\mathcal L^*$ in \eqref{lyap_op_adjoint}. The term $2\mu_k v_k^\top A v_k$ is small since $\mu_k$ is small. Furthermore, $v_k^\top \mathcal N (I\otimes Q_T)\mathcal N^\top v_k$ and $v_k^\top \mathcal H (P_T\otimes Q_T)\mathcal H^\top v_k$ are nonnegative. By the monotone convergence theorem, we also obtain
\begin{align*}
\mu_k = v_k^\top Q_T v_k
= \sum_{i=1}^\infty \int_0^T v_k^\top \mathcal Z_{i,T}(s)v_k \, \dd s= \int_0^T \sum_{i=1}^\infty  v_k^\top \mathcal Z_{i,T}(s)v_k \, \dd s.
\end{align*}
Therefore, we expect $\sum_{i=1}^\infty v_k^\top \mathcal Z_{i,T}(T)v_k$ to be small as well. If $Q$ is considered instead, this series is not present. Moreover, all other arguments are still valid. Consequently, $v_k^\top C^\top C v_k$ is expected to be small if $v_k$ is an eigenvector of $Q_T$ or $Q$, respectively. This suggests that it is beneficial if the Gramians coincide, i.e., $P_T = Q_T$ or $P = Q$, since this allows both $v_k^\top P_T v_k$ (or $v_k^\top P v_k$) and $v_k^\top C^\top C v_k$ to be small when $v_k$ corresponds to a small eigenvalue.

This observation motivates the algorithm proposed below. In fact, in this case the Gramians will additionally be diagonal, so that $(v_k)$ coincides with the canonical basis of $\mathbb R^n$, which simplifies the truncation of irrelevant variables.
The bounds in \eqref{dom_sub_y_PT} and \eqref{dom_sub_y_P} further suggest that alternative choices of observability Gramians are possible, provided they guarantee that $v_k^\top C^\top C v_k$ is sufficiently small. While this is indeed the case in principle, an additional key requirement is that the observability Gramians admit a simultaneous diagonalization with the reachability Gramians via a state-space transformation, as this property is essential for our reduction algorithm. The suitability of the choices $Q_T$ and $Q$ will become apparent in the following two sections.

\subsection{Exact reduced-order models}\label{sec_exact_rom}

Before we study the above mentioned state-space transformation, we investigate exact reduced systems based on Gramians. This is because a state-space transformation requires the Gramians to be invertible. This can always be achieved by removing the kernels of these generally positive semidefinite matrices. We will see that this reduction comes without any error. In fact, this aspect was touched in \cite{morBenG24}, where unreachable and unobservable states were linked to kernels of Gramians. However, the exact reduced systems were not worked out.
First, let
\begin{align}\label{imP_basis}
V=\begin{bmatrix}                                                                                                                                                                                                                                                                p_1 & \dots & p_{\hat n}                                                                                                                                                                                                                                                                                                                                                                                                  \end{bmatrix}
\end{align}
where $p_1, \dots,  p_{\hat n}$ are all orthonormal eigenvectors of either $P_T$ or $P$ associated with all non-zero eigenvalues and $\hat n\leq n$.
\begin{prop}\label{prop_remove_kerP}
We assume that $P_T$ (or $P$) exists.
Moreover, suppose that \eqref{stochstatenew} has a (unique) global solution on $[0, T]$ and let $V$ be the matrix defined in \eqref{imP_basis}. Then, we have $x(t)=V \hat x(t)$, where  $\hat x(t)= \begin{bmatrix}
 \langle x(t), p_1\rangle &\dots & \langle x(t), p_{\hat n}\rangle
  \end{bmatrix}^\top\in\mathbb R^{\hat n}$ satisfies \begin{subequations}\label{red_sys_zero_error}\begin{align}\label{redstate_zero_error}
             \dot {\hat x}(t)&=\hat A \hat x(t)+ \hat B u(t)+ \hat N(\mathfrak u(t)\otimes \hat x(t))+ \hat H(\hat x(t)\otimes \hat x(t)),\quad \hat x(0) = 0,\\ \label{output_red_zero_error}
            y(t) &= \hat C \hat x(t),\quad t\in [0, T].
\end{align}
\end{subequations}
The output $y$ is the same as the one in \eqref{output_eq} and \begin{align*}
\hat A=V^\top A V, \quad \hat B= V^\top B,\quad \hat N= V^\top N (I \otimes V), \quad \hat H=V^\top H (V \otimes V), \quad \hat C=C V.
        \end{align*}
\end{prop}
\begin{proof}
We show the result using $P_T$ only as the arguments for $P$ are identical. Now, let $v\in\kernel(P_T)$ be arbitrary. Multiplying \eqref{eq_for_PT} with $v^\top$ from the left and with $v$ from the right, we obtain \begin{align}\label{some_proof_eq0}
v^\top BB^\top v +v^\top N (I\otimes P_T) N^\top v
+ v^\top H (P_T\otimes P_T)H^\top v
=  \sum_{i=1}^\infty v^\top Z_{i,T}(T)v.
\end{align}
We know that $0=v^\top P_T v=\sum_{i=1}^\infty \int_0^T v^\top Z_{i,T}(s)v\,\dd s$. This implies that $v^\top Z_{i,T}(T) v=0$ for all $i\geq 1$ meaning that the right-hand side in \eqref{some_proof_eq0} disappears. As all terms on the left-hand side of \eqref{some_proof_eq0} are nonnegative, we find that $B^\top v=0$ and $(I\otimes P_T) N^\top v=0$. We further have that
$(P_T\otimes P_T)H^\top v=0$. We apply this result when multiplying \eqref{eq_for_PT} with $v$ from the right leading to $P_T A^\top v=0$. We use the right-hand side of \eqref{stochstatenew} to define $f(t, x) = Ax+Bu(t)+ N (\mathfrak u(t)\otimes x)+H(x\otimes x)$. Given that $x\in\im(P_T)$ and exploiting the above identities, we have \begin{align*}
f(t, x)^\top v=  x^\top A^\top v+u(t)^\top B^\top v+ (\mathfrak u(t)^\top \otimes x^\top)N^\top v +(x^\top\otimes x^\top)H^\top v    = 0                                                                                                                                                                                                                                                                                                                                                                                                                                                                                                                                                                                                            \end{align*}
meaning that $f(t, x)\in\im(P_T)$ for $t\in [0, T]$ if $x\in\im(P_T)$. Let us now consider the equation\begin{align}\label{ODE_z}
  \dot z(t)= f_{P_T}(t, z(t)), \quad z(0)=0,                                                                                                                                                                                                             \end{align}
where $f_{P_T}:\im(P_T)\to \im(P_T)$ is the restriction of $f$ on $\im(P_T)$. By Carath\' eodory's existence theorem, there is a local solution $z$ for \eqref{ODE_z} on $[0, \delta)$ for some $\delta>0$. Since $f(t, x)=f_{P_T}(t, x)$ for $x\in\im(P_T)$, $z$ is also a solution of \eqref{stochstatenew}. Therefore, $x\equiv z$ on $[0, \delta)$, so that $x(t)\in\im(P_T)$ for $t\in [0, \delta)$. We define \begin{align*}
 \tau:=\sup\{t\in[0, T]: x(s)\in\im(P_T), s\in[0, t]\}.                                                                                                                                                                                                                                                                                                                                                                                                                                                                                                                                                                                                                                                                                                                                                                                                                                        \end{align*}
 We know that $\tau\geq \delta>0$. Let us now assume that $\tau<T$. Since $x$ is continuous and $\im(P_T)$ is closed, we have $x(\tau)\in\im(P_T)$. Considering
 \begin{align*}
  \dot z(t)= f_{P_T}(t, z(t)), \quad z(\tau)=x(\tau)\in\im(P_T),                                                                                                                                                                                                             \end{align*}
we can now conclude that $x(t)=z(t)\in\im(P_T)$ for $t\in[\tau, \tau+\epsilon)$ for some $\epsilon>0$. This contradicts the definition of $\tau$. Hence, we have $x(t)\in\im(P_T)$ for all $t\in[0, T]$. As the columns of $V$ represent an orthonormal basis of $\im(P_T)$, we obtain  \begin{align}\label{someest}
x(t)=\sum_{k=1}^{\hat n}\langle x(t), p_k\rangle\, p_k = V \hat x(t).                                                                                                                                                                                                                                                                                                                                                                                                                                                                                                                                                                                                                                                                                                                                                                                                                                                                                                                                                                                                                                                                                                                                                                                                                                                                                                                                                                                                                                                                                                                                                                                                                                                                                                                                                                         \end{align}
Inserting \eqref{someest} into \eqref{original_system} and multiplying the resulting state equation with $V^\top$ from the left, the result follows.
\end{proof}
We briefly ceck that the reachability Gramian of the exact reduced system \eqref{red_sys_zero_error} is positive definite.
\begin{prop}\label{prop_inv_reach_gram}
Assuming the existence of the reachability Gramians, let $V$ be the matrix of orthogonal eigenvectors of $P_T$ (or $P$) defined in \eqref{imP_basis}. Then, system \eqref{red_sys_zero_error} has an invertible reachability Gramian $D=V^\top P_T V$ (or $D=V^\top P V$), where $D$ is the diagonal matrix containing all non-zero eigenvalues of  $P_T$ (or $P$).
\end{prop}
\begin{proof}
We conduct this proof only using $P_T$. The steps for $P$ are identical. Now, let $v\in \kernel(P_T)$ be arbitrary, then \begin{align*}
 0= v^\top P_T v= \sum_{i=1}^\infty\int_0^T \|Z_{i, T}(s)^{\frac{1}{2}} v\|^2\dd s                                                                                                                                                                                                                                                                                                                                                                                                                                                                                                                                                                                                                                                                                                                                                                                                                                                                                                                                                                                                                                                                                                                                                                                                                                                                                                                                                                                                                                                                                                                                                                                                                                                                                                                                                                                                                                                                                                                                                                                                                                                                                                                                                                                                                                                                                                                                                                                                                                                                                                                                                                              \end{align*}
 using that $Z_{i, T}$ is symmetric prositive semidefinite. Consequently, $v^\top Z_{i, T}(t)=0$ for all $t\in [0, T]$ and all $i\geq 1$. Therefore, the columns of $Z_{i, T}(t)$ take values in $\im(V)=\im(P_T)$. This means that we find matrices $M_i(t)$, so that $Z_{i, T}(t)=V M_i(t) V^\top$. Inserting this identities into \eqref{eq_for_Z1} and \eqref{eq_for_Z2} as well as multiplying the results with $V^\top$ from the left and with $V$ from the right, we see that $\sum_{i=1}^\infty\int_0^T M_i(s)\dd s = V^\top P_T V$ is the time-limited reachability Gramian of \eqref{red_sys_zero_error}. By the eigenvalue decomposition of $P_T$, we obtain $P_T= V D V^\top$ and the result follows.
\end{proof}
\begin{remark}
 Due to Proposition \ref{prop_inv_reach_gram} we can always assume without loss of generality that $P_T$ and $P$ are invertible. This is already crucial for the next step, where the removal of the kernel of an observability Gramian is discussed.
\end{remark}
 Now, let us introduce
\begin{align}\label{imQ_basis}
W=\begin{bmatrix}                                                                                                                                                                                                                                                                q_1 & \dots & q_{\hat {\mathfrak n}}                                                                                                                                                                                                                                                                                                                                                                                                  \end{bmatrix}
\end{align}
 where $q_1, \dots,  q_{\hat{\mathfrak n}}$ are all orthonormal eigenvectors of either $Q_T$ or $Q$ associated with all non-zero eigenvalues and $\hat {\mathfrak n}\leq n$.
%  For simplicity of the notation, $\hat n$ also denotes the number of non-zero eigenvalues of the observability Gramians. The number of zero eigenvalues of reachability and observability Gramians  generally differ.
\begin{prop}\label{prop_remove_kerQ}
Suppose that $Q_T$ (or $Q$) exists and that \eqref{stochstatenew} has a global solution on $[0, T]$. Further, let $W$ be the matrix of orthonormal eigenvectors of $Q_T$ (or $Q$) defined in \eqref{imQ_basis}. Moreover, assume w.l.o.g. that $P_T$ (or $P$) is invertible. Suppose that $H$ is symmetric meaning that $H(v\otimes z)=H(z\otimes v)$ for all $v, z\in\mathbb R^n$. Then, there exsit a state variable $\hat x$ with $y(t)= W \hat x(t)$ and \begin{align}\label{zero_error_q}
 \dot {\hat x}(t)&=\hat A \hat x(t)+ \hat B u(t)+ \hat N(\mathfrak u(t)\otimes \hat x(t))+ \hat H(\hat x(t)\otimes \hat x(t)),\quad \hat x(0) = 0                                                                                                                                                                                                                                                                                                                                                         \end{align}
 for $t\in[0, T]$, where
 \begin{align}\label{zero_error_q_mat}
 \hat A=W^\top A W, \quad \hat B= W^\top B,\quad \hat N= W^\top N (I \otimes W), \quad \hat H=W^\top H (W \otimes W).
 \end{align}
%  In addition, system \eqref{zero_error_q} has a reachability Gramian $W^\top P_T W$ (or $W^\top P W$) and an invertible observability Gramian $\mathcal D=W^\top Q_T W$ (or $\mathcal D=W^\top Q W$),  where $\mathcal D$ is the diagonal matrix containing all non-zero eigenvalues of  $Q_T$ (or $Q$).
\end{prop}
\begin{proof}
We show the result using $Q_T$ only as the arguments for $Q$ are identical. Let $q_k$, $k>\hat{\mathfrak n}$, be the orthogonal eigenvectors of $Q_T$ corresponding to zero eigenvalues. Setting $v_k=q_k$, $k>\hat{\mathfrak n}$, in \eqref{dominat_direction_q}, we obtain \begin{align}\label{some_proof_eq}
q_k^\top C^\top C q_k +q_k^\top \mathcal N (I\otimes Q_T) \mathcal N^\top q_k
+ q_k^\top \mathcal H (P_T\otimes Q_T)\mathcal H^\top q_k
=  \sum_{i=1}^\infty q_k^\top \mathcal Z_{i,T}(T)q_k.
\end{align}
We know that $0=q_k^\top Q_T q_k=\sum_{i=1}^\infty \int_0^T q_k^\top\mathcal Z_{i,T}(s)q_k\dd s$. This implies that $q_k^\top\mathcal Z_{i,T}(T)q_k=0$ for all $i\geq 1$ meaning that the right-hand side in \eqref{some_proof_eq} disappears. As all terms on the left-hand side of \eqref{some_proof_eq} are nonnegative, we find that $Cq_k=0$ and $(I\otimes Q_T) \mathcal N^\top q_k=0$ for all $k>\hat{\mathfrak n}$. The last equation further implies that $Q_T N (I\otimes q_k)=0$ for all $k>\hat{\mathfrak n}$. We further have that
$(P_T\otimes Q_T)\mathcal H^\top q_k=0$ being equivalent to $(I\otimes Q_T)\mathcal H^\top q_k=0$ for all $k>\hat{\mathfrak n}$ using that $P_T$ is invertible. This implies that $Q_T H (I\otimes q_k)=0$ for all $k>\hat{\mathfrak n}$. We apply this result when multiplying \eqref{eq_for_QT} with $q_k$ from the right leading to $Q_T A q_k=0$ for all for all $k>\hat{\mathfrak n}$. With the above results in mind, we obtain  \begin{align*}
C x(t)=\sum_{k=1}^n\langle x(t), q_k\rangle\, C q_k=C\sum_{k=1}^{\hat {\mathfrak n}}\langle x(t), q_k\rangle\, q_k = CW \hat x(t),                                                                                                                                                                                                                                                                                                                                                                                                                                                                                                                                                                                                                                                                                                                                                                                                                                                                                                                                                                                                                                                                                                                                                                                                                                                                                                                                                                                                                                                                                                                                                                                                                                                                                                                                                                            \end{align*}
where $\hat x(t)= \begin{bmatrix}
 \langle x(t), q_1\rangle &\dots & \langle x(t), q_{\hat {\mathfrak n}}\rangle
  \end{bmatrix}^\top\in\mathbb R^{\hat n}$. Moreover, we have $W^\top A x(t)=W^\top A W \hat x(t) + \sum_{k=\hat {\mathfrak n}+1}^{n}\langle x(t), q_k\rangle\, W^\top A q_k= W^\top A W \hat x(t)$ using that $A q_k\in\ker(Q_T)$ for $k>\hat{\mathfrak n}$. In the same spirit, we obtain $W^\top N (\mathfrak u(t)\otimes x(t))=W^\top N (I\otimes x(t)) \mathfrak u(t)=W^\top N (I\otimes W\hat x(t)) \mathfrak u(t)+W^\top N \Big(I\otimes \sum_{k=\hat {\mathfrak n}+1}^{n}\langle x(t), q_k\rangle\,q_k\Big) \mathfrak u(t)= W^\top N (\mathfrak{u}(t)\otimes W\hat x(t))= W^\top N(I\otimes W) (\mathfrak{u}(t)\otimes \hat x(t))$, since $N (I\otimes q_k)\in\ker(Q_T)$ for all $k>\hat{\mathfrak n}$. The same way, it holds that $W^\top H (x(t)\otimes x(t))= W^\top H(x(t)\otimes W\hat x(t))$, since $H(I\otimes q_k)\in\ker(Q_T)$ for all $k>\hat{\mathfrak n}$. Given that $H$ is symmetric, we obtain $W^\top H (x(t)\otimes x(t))= W^\top H(W\hat x(t)\otimes x(t))= W^\top H(W\hat x(t)\otimes W \hat x(t))= W^\top H(W\otimes W)(\hat x(t)\otimes \hat x(t))$. Exploiting that $\hat x(t) = W^\top x(t)$, \eqref{zero_error_q} follows from multiplying \eqref{stochstatenew} with $W^\top$ from the left.
\end{proof}
\begin{remark}
In the next subsection, we assume that the Gramians are invertible. In pratice, this is often not given at first. However, we can determine a reduced system like in Proposition \ref{prop_remove_kerP} to achieve a positive definite reachability Gramian. Subsequently, the reduced-order model of Proposition \ref{prop_remove_kerQ} can be computed. As a consequence the observability Gramian will be invertible as well. As the approximations in both propositions lead to exact reduced systems, the positivity assumption on the Gramians is no further restriction. It is also interesting to notice that for the the exact reduced systems in Propositions \ref{prop_remove_kerP} and \ref{prop_remove_kerQ}, no further restrictions on the control but $u, \mathfrak{u}\in L^2_T$ are required. This is in contrast to the dominant subspace arguments in Section \ref{sec_dom_sub}, where \eqref{def_fuu} is assumed.
\end{remark}

\subsection{State-space transformation}
Let $S\in\mathbb R^{n\times n}$ be an invertible matrix and define $\tilde x = S x$. Inserting $x=S^{-1}\tilde x$ into \eqref{original_system} and multiplying with $S$ from the right-hand side yields \begin{subequations}\label{trans_sys}
\begin{align}\label{transstate}
             \dot {\tilde x}(t)&=\tilde A \tilde x(t)+\tilde Bu(t)+ \tilde N(\mathfrak u(t)\otimes \tilde x(t))+\tilde H(\tilde x(t)\otimes \tilde x(t)),\quad \tilde x(0) = 0,\\ \label{output_eq2}
            y(t) &= \tilde C \tilde x(t),\quad t\in [0, T],
\end{align}
\end{subequations}
where $\tilde A= S AS^{-1}$, $\tilde B= S B$, $\tilde N= S N (I\otimes S^{-1})$, $\tilde H= SH (S^{-1}\otimes S^{-1})$ and $\tilde C=CS^{-1}$. Below, we also exploit the block structure of $\tilde N:=\begin{bmatrix}\tilde N_1 & \ldots & \tilde N_m\end{bmatrix}$,  $\tilde H:=\begin{bmatrix}\tilde H_1 & \ldots & \tilde H_n\end{bmatrix}$, where $\tilde N_k, \tilde H_j\in\mathbb R^{n\times n}$, $k\in\{1, \dots, m\}$ and $j\in\{1, \dots, n\}$.
\begin{prop}\label{prop_reach_gram_trans}
 Suppose that the reachability Gramians $P_T:=\sum_{i=1}^\infty\int_0^T Z_{i, T}(s)\dd s$ and    $P:=\sum_{i=1}^\infty\int_0^\infty Z_{i, \infty}(s)\dd s$ of the original system                                                                                                                                                                                                                                                                                                                                                                                                                                                  \eqref{original_system} exist. Then, the reachability Gramians $\tilde P_T$ and $\tilde P$ of \eqref{trans_sys} are given by \begin{align*}
 \tilde P_T=S P_T S^{\top} \quad\text{and}\quad \tilde P=S P S^{\top}.                                                                                                                                                                                                                                                                                                                                                                                                                                                                                                                                                                                                                                                                                                                                                          \end{align*}
\end{prop}
\begin{proof}
 We only prove the representation for $\tilde P_T$ and omit the proof for $\tilde P$, because all steps are identical.  By definition, we have that $\tilde P_T:=\sum_{i=1}^\infty\int_0^T \tilde Z_{i, T}(s)\dd s$, where $\tilde Z_{i, T}$ are the solutions of \eqref{eq_for_Z1} and \eqref{eq_for_Z2} when replacing $(A, B, N, H)$ by  $(\tilde A, \tilde B, \tilde N, \tilde H)$. Multiplying \eqref{eq_for_Z1} with $S$ from the left and with $S^\top$ from the righ-hand side and using the definition of $\mathcal L$ in \eqref{lyap_op}, we obtain \begin{align*}
 \frac{\dd}{\dd t} SZ_{1, T}(t)S^\top &= \tilde A (SZ_{1, T}(t)S^\top)+(SZ_{1, T}(t)S^\top){\tilde A}^\top + \tilde N\big(I\otimes (SZ_{1, T}(t)S^\top)\big) \tilde N^\top,\\
 SZ_{1, T}(0)S^\top &= \tilde B{\tilde B}^\top.
\end{align*}
Thus, we see that $\tilde Z_{1, T}(t)=SZ_{1, T}(t)S^\top$. Let us now assume that $\tilde Z_{k, T}(t)=SZ_{k, T}(t)S^\top$ for $k\in\{1, \dots, i-1\}$. If we multiply \eqref{eq_for_Z2} with $S$ and $S^\top$ from the left and the right, respectively, we obtain \begin{align*}
 \frac{\dd}{\dd t} SZ_{i, T}(t)S^\top &= \tilde A (SZ_{i, T}(t)S^\top)+(SZ_{i, T}(t)S^\top){\tilde A}^\top + \tilde N\big(I\otimes (SZ_{i, T}(t)S^\top)\big) \tilde N^\top,\\
 SZ_{i, T}(0)S^\top &= S H \big(\sum_{k=1}^{i-1}\int_0^T Z_{k, T}(s)\dd s\otimes \int_0^T Z_{i-k, T}(s) \dd s\big) H ^\top S^\top\\
 &= S H \big(\sum_{k=1}^{i-1}S^{-1}\int_0^T \tilde Z_{k, T}(s)\dd s\, S^{-\top}\otimes S^{-1}\int_0^T \tilde Z_{i-k, T}(s) \dd s\, S^{-\top}\big) H ^\top S^\top\\
 &= \tilde H \big(\sum_{k=1}^{i-1}\int_0^T \tilde Z_{k, T}(s)\dd s\otimes \int_0^T \tilde Z_{i-k, T}(s) \dd s\big) \tilde H ^\top.
 \end{align*}
This shows that $\tilde Z_{i, T}(t)=SZ_{i, T}(t)S^\top$ for all $i\geq 1$ and hence $\tilde P_T=SP_TS^\top$.
\end{proof}
\begin{prop}\label{prop_trans_Q}
 Suppose that the observability Gramians $Q_T:=\sum_{i=1}^\infty\int_0^T \mathcal Z_{i, T}(s)\dd s$ and $Q:=\sum_{i=1}^\infty\int_0^\infty \mathcal Z_{i, \infty}(s)\dd s$ of the original system                                                                                                                                                                                                                                                                                                                                                                                                                                                  \eqref{original_system} exist. Then, the observability Gramians $\tilde Q_T$ and $\tilde Q$ of \eqref{trans_sys} are given by \begin{align*}
 \tilde Q_T=S^{-\top} Q_T S^{-1} \quad\text{and}\quad \tilde Q=S^{-\top} P S^{-1}.                                                                                                                                                                                                                                                                                                                                                                                                                                                                                                                                                                                                                                                                                                                                                          \end{align*}
\end{prop}
\begin{proof}
 We only prove the representation for $\tilde Q_T$ and omit the proof for $\tilde Q$, due to the similarity.  By definition, we have that $\tilde Q_T:=\sum_{i=1}^\infty\int_0^T \tilde {\mathcal Z}_{i, T}(s)\dd s$, where $\tilde {\mathcal Z}_{i, T}$ are the solutions of \eqref{eq_for_Z1_ad} and \eqref{eq_for_Z2_ad} when replacing $(A, B, N, H, C)$ by  $(\tilde A, \tilde B, \tilde N, \tilde H, \tilde C)$ as well as $(\mathcal N, \mathcal H)$ by $(\tilde {\mathcal N}, \tilde {\mathcal H})$, where $\tilde {\mathcal N}$ and $\tilde {\mathcal H}$ are the block-wise transposed of $\tilde N$ and $\tilde H$. More precisely, we have \begin{align*}
 \tilde N&=  S N (I\otimes S^{-1}) = \begin{bmatrix}S N_1 S^{-1}  & \ldots & S N_m S^{-1} \end{bmatrix},\\
 \tilde H&= SH (I\otimes S^{-1}) (S^{-1}\otimes I)  = \begin{bmatrix}S H_1 S^{-1}  & \ldots & S H_n S^{-1} \end{bmatrix} (S^{-1}\otimes I)                                                                                                                                                                                                                                                                                                                                                                                                                                                                                                                                                                  \end{align*}
 and hence \begin{align*}
 \tilde {\mathcal N}&= \begin{bmatrix}S^{-\top} N_1^\top S^{\top}  & \ldots & S^{-\top}  N_m^\top S^{\top} \end{bmatrix}= S^{-\top} \mathcal N (I\otimes S^{\top}),\\
 \tilde {\mathcal H}&=  \begin{bmatrix}S^{-\top} H_1^\top S^{\top}  & \ldots & S^{-\top} H_n^\top S^{\top} \end{bmatrix} (S^{-1}\otimes I)=S^{-\top}\mathcal H (S^{-1}\otimes S^{\top}).                                                                                                                                                                                                                                                                                                                                                                                                                                                                                                                                                                  \end{align*}
 Multiplying \eqref{eq_for_Z1_ad} with $S^{-\top}$ from the left and with $S^{-1}$ from the righ-hand side and using the definition of $\mathcal L^*$ in \eqref{lyap_op_adjoint}, we obtain \begin{align*}
\frac{\dd}{\dd t} S^{-\top}{\mathcal Z}_{1, T}(t)S^{-1} &= \tilde A^\top (S^{-\top} \mathcal Z_{1, T}(t)S^{-1})+(S^{-\top}\mathcal Z_{1, T}(t)S^{-1}){\tilde A} + \tilde{\mathcal N}\big(I\otimes (S^{-\top}\mathcal Z_{1, T}(t)S^{-1})\big) \tilde {\mathcal N}^\top,\\
 S^{-\top}\mathcal Z_{1, T}(0)S^{-1} &= \tilde C^\top{\tilde C}.
\end{align*}
Consequently, we have $\tilde {\mathcal Z}_{1, T}(t)=S^{-\top}\mathcal Z_{1, T}(t)S^{-1}$. We assume that $\tilde {\mathcal Z}_{k, T}(t)=S^{-\top}\mathcal Z_{k, T}(t)S^{-1}$ for $k\in\{1, \dots, i-1\}$. By the proof of  Proposition \ref{prop_reach_gram_trans}, we also know that $\tilde {Z}_{k, T}(t)=S Z_{k, T}(t)S^{\top}$ for $k\in\{1, \dots, i-1\}$.
We multiply \eqref{eq_for_Z2_ad} with $S^{-\top}$ and $S^{-1}$ from the left and the right, respectively. Therefore, we obtain \begin{align*}
 \frac{\dd}{\dd t} S^{-\top}{\mathcal Z}_{i, T}(t)S^{-1} &= \tilde A^\top (S^{-\top} \mathcal Z_{i, T}(t)S^{-1})+(S^{-\top}\mathcal Z_{i, T}(t)S^{-1}){\tilde A} + \tilde{\mathcal N}\big(I\otimes (S^{-\top}\mathcal Z_{i, T}(t)S^{-1})\big) \tilde {\mathcal N}^\top,\\
 S^{-\top} \mathcal Z_{i, T}(0)S^{-1} &= S^{-\top} \mathcal H \big(\sum_{k=1}^{i-1}\int_0^T Z_{k, T}(s)\dd s\otimes \int_0^T \mathcal Z_{i-k, T}(s) \dd s\big) \mathcal H^\top S^{-1}\\
 &= S^{-\top} \mathcal H \big(\sum_{k=1}^{i-1}S^{-1}\int_0^T \tilde Z_{k, T}(s)\dd s\, S^{-\top}\otimes S^{\top}\int_0^T \tilde {\mathcal Z}_{i-k, T}(s) \dd s\, S\big) \mathcal H^\top S^{-1}\\
 &= \tilde {\mathcal H} \big(\sum_{k=1}^{i-1}\int_0^T \tilde Z_{k, T}(s)\dd s\otimes \int_0^T \tilde {\mathcal Z}_{i-k, T}(s) \dd s\big) \tilde {\mathcal H}^\top.
 \end{align*}
This shows that $\tilde {\mathcal Z}_{i, T}(t)=S^{-\top} \mathcal Z_{i, T}(t)S^{-1}$ for all $i\geq 1$ and hence $\tilde Q_T=S^{-\top} Q_T S^{-1}$.
 \end{proof}
We are now discussing the balancing transformation that simultaneously diagonalizes reachability and observability Gramians. This transformation requires the Gramians to be invertible. This can always be assumed without loss of generality due to Section \ref{sec_exact_rom}.
\begin{prop}\label{prop_bal}
 Let $P_T$ and $Q_T$ be the Gramians of \eqref{original_system} according to Definitions \ref{def_P} and \ref{def_Q}, respectively. Given that $P_T$ and $Q_T$ are invertible, we obtain $\tilde P_T=\tilde Q_T = \Sigma= \diag(\sigma_1,\ldots,\sigma_n)$ using  \begin{align}\label{bal_transform_new}
       S=\Sigma^{\frac{1}{2}} V^\top L_P^{-1},                                                                                                                                                                                                                                                                                                                                                                                                                                                                                                                                                                                                                                                                                                                                                                                          \end{align}
where $P_T=L_PL_P^\top$ and $L_P^\top Q_T L_P=V\Sigma^2 V^\top$ is an eigenvalue decomposition with an orthogonal $V$. The diagonal entries $\sigma_1,\dots, \sigma_n$ of $\Sigma$ are the square root of the eigenvalues of $P_T Q_T$. The matrix $S$ in \eqref{bal_transform_new} is called balancing transformation.
\end{prop}
\begin{proof}
 We exploit Propositions \ref{prop_reach_gram_trans} and \ref{prop_trans_Q}. The balancing transformation in \eqref{bal_transform_new} is then inserted into $\tilde P_T= SP_T S^\top= \Sigma^{\frac{1}{2}} V^\top L_P^{-1} P_T L_P^{-\top} V \Sigma^{\frac{1}{2}}=\Sigma$ and use the same $S$ for $\tilde Q_T=S^{-\top}Q_T S^{-1}= \Sigma^{-\frac{1}{2}} V^\top L_P^{\top} Q L_P V \Sigma^{-\frac{1}{2}}=\Sigma$. Further, the matrix $L_P^\top Q_T L_P$ has eigenvalues $\sigma_1^2, \dots, \sigma_n^2$, but their eigenvalues coincide with the ones of $L_PL_P^\top Q_T= P_TQ_T$. This concludes the proof.
\end{proof}
In the following, the diagonal entries of $\Sigma$ are called Hankel singular values.
\begin{remark}
The result of Proposition \ref{prop_bal} also holds when $P_T$ is replaced by $P$ and $Q$ is used instead of $Q_T$.
\end{remark}

\subsection{(Inexact) reduced-order model by balanced truncation}

Let $S$ be the balancing transformation in \eqref{trans_sys}. Let us decompose the balanced realization as follows: \begin{align}\label{matrix_decomp}
\tilde A= \begin{bmatrix}{A}_{11}&A_{12}\\ A_{21}& A_{22}\end{bmatrix},\quad \tilde B = \begin{bmatrix}{B}_1\\ B_2\end{bmatrix}, \quad \tilde N_k= \begin{bmatrix}{N}_{k,11}&N_{k, 12}\\ N_{k, 21}& N_{k, 22}\end{bmatrix},\quad \tilde H_j= \begin{bmatrix}{H}_{j,11}&H_{j, 12}\\ H_{j, 21}& H_{j, 22}\end{bmatrix},\quad \tilde C = \begin{bmatrix}{C}_1& C_2\end{bmatrix},                                                      \end{align}
where ${A}_{11}, N_{k, 11}, H_{j, 11}\in\mathbb R^{r\times r}$, ${B}_1\in\mathbb R^{r\times m}$, $C_1\in\mathbb R^{p\times r}$ for $k\in\{1, \dots, m\}$ and $j\in\{1, \dots, n\}$. The balanced state is further decomposed into \begin{align}\label{state_decomp}
 \tilde x   = \begin{bmatrix}{x}_1\\ x_2\end{bmatrix} \quad \text{with}\quad x_1(t)\in\mathbb R^r, \quad       x_2(t)\in\mathbb R^{n-r},                                                                                                                                                                                          \end{align}
where the $n-r$ variable $x_2$ are associated with small (and hence truncated) Hankel singular values $\sigma_{r+1}, \dots, \sigma_n$.
Inserting \eqref{matrix_decomp} and \eqref{state_decomp} into \eqref{trans_sys} the reduced system is obtained by setting $x_2\equiv 0$ in \eqref{output_eq2} as well as in the differential equation of $x_1$ in \eqref{transstate}. The intuition is that $x_2$ has a low impact on $y$ and it is supposed to be small making it negligible in the dynamics of $x_1$. In addition, the differential equation for $x_2$ is truncated in \eqref{transstate}. This leads to \begin{subequations}\label{red_sys}
\begin{align}\label{redstate}
             \dot {\hat x}(t)&=\hat A \hat x(t)+ \hat B u(t)+ \hat N(\mathfrak u(t)\otimes \hat x(t))+ \hat H(\hat x(t)\otimes \hat x(t)),\quad \hat x(0) = 0,\\ \label{output_red}
            \hat y(t) &= \hat C \hat x(t),\quad t\in [0, T],
\end{align}
\end{subequations}
setting \begin{align}\label{bt_matrices}
\hat A=A_{11}, \quad \hat B= B_1,\quad \hat N= N_{11}, \quad \hat H=H_{11}, \quad \hat C=C_1
        \end{align}
with $N_{11}:=\begin{bmatrix} N_{1, 11} & \ldots & N_{m, 11}\end{bmatrix}$ and  $H_{11}:=\begin{bmatrix}H_{1, 11} & \ldots & H_{r, 11}\end{bmatrix}$.

\subsection{Error analysis}

Error bounds for dimension reduction applied to  quadratic bilinear systems could not be achieved so far. We provide bounds that are interesting in the context of balanced truncation \cite{morBenG24} and $\mathcal{H}_2$-optimal model reduction \cite{nonlinear_irka}.

\subsubsection{General Gramian-based bound}\label{sec_gen_bound}

Let us determine a general output bound between \eqref{original_system} and \eqref{red_sys}. In this context, let us assume that the reachability Gramians of both systems exist. We can write \begin{align*}
y(t)-\hat y(t) = Cx(t)-\hat C\hat x(t)  = \begin{bmatrix} C & -\hat C\end{bmatrix}   \begin{bmatrix} x(t) \\ {\hat x}(t)\end{bmatrix}  = C_e  x_e(t)                                                                                                                                                                                                                             \end{align*}
defining $x_e(t)=\begin{bmatrix} x(t) \\ {\hat x}(t)\end{bmatrix}$ and $C_e=\begin{bmatrix} C & -\hat C\end{bmatrix}$. One can easily check that the state $x_e$ satisfies \begin{align*}
  \dot x_e(t)&=A_e x_e (t)+B_e u(t)+ N_e (\mathfrak u(t)\otimes x_e(t))+H_e(x_e(t)\otimes x_e(t)),\quad x_e(0) = 0,                                                                                                                                                                                                                                                                                                                                                    \end{align*}
where \begin{align*}
A_e&=\begin{bmatrix}{A}&0\\ 0& {\hat A}\end{bmatrix}, \quad B_e=\begin{bmatrix}B\\ {\hat B}\end{bmatrix}, \quad N_e=\begin{bmatrix} N_{e, 1}&\dots & N_{e, m}\end{bmatrix} \text{ with }N_{e, k}=\begin{bmatrix}{N_k}&0\\ 0& {\hat N}_k\end{bmatrix} \text{and}\\
H_e&=\begin{bmatrix}H ([I_n \,0]\otimes [I_n \,0])\\ {\hat H}([0 \,I_r]\otimes [0 \, I_r])\end{bmatrix}=\begin{bmatrix} H_{e, 1}&\dots & H_{e, n+r}\end{bmatrix}  \text{ with } H_{e, i}= \begin{bmatrix}{H_i}&0\\ 0& 0\end{bmatrix}, \quad H_{e, n+j}= \begin{bmatrix}{0}&0\\ 0& {\hat H_j}\end{bmatrix}
      \end{align*}
for $k\in\{1, \dots, m\}$, $i\in\{1, \dots, n\}$ and $j\in\{1, \dots, r\}$. Note that the system solved by $x_e$ was derived in a similar manner in \cite{nonlinear_irka}. Applying Theorem \ref{thm_gramian_nased bound} to the above error system and assuming that
\eqref{def_fuu} holds (i.e. $f(u, \mathfrak{u})<\infty$),
we obtain
 \begin{align*}
\sup_{t\in[0, T]}\left\|y(t)-\hat y(t)\right\|
\leq  \sqrt{\trace(C_e P_{e, T} C_e^\top)}\sqrt{f(u, \mathfrak{u})}\leq  \sqrt{\trace(C_e P_{e} C_e^\top)}\sqrt{f(u, \mathfrak{u})},
\end{align*}
where $P_{e, T}:=\sum_{i=1}^\infty\int_0^T Z_{e, i, T}(s)\dd s$ and   $P_e:=\sum_{i=1}^\infty\int_0^\infty Z_{e, i, \infty}(s)\dd s$ with \begin{align}\label{eq_for_Z1_error}
 \dot {Z}_{e, 1, T}(t) = \mathcal L_e\big(Z_{e, 1, T}(t)\big),\quad Z_{e, 1, T}(0) = B_eB_e^\top
\end{align}
with $\mathcal L_e$ defined like $\mathcal L$ in \eqref{lyap_op} replacing $(A, N)$ by $(A_e, N_{e})$.
In addition, $Z_{e, i, T}(t)$, $t\geq 0$,
($i\geq 2$) are the solutions of \begin{align}\label{eq_for_Z2_error}
 \dot {Z}_{e, i, T}(t) = \mathcal L_e\big(Z_{e, i, T}(t)\big),\quad Z_{e, i, T}(0) = H_e\big(\sum_{k=1}^{i-1}\int_0^T Z_{e, k, T}(s)\dd s\otimes \int_0^T Z_{e, i-k, T}(s) \dd s\big)H_e^\top.
\end{align}
Replacing $T$ by $\infty$ in \eqref{eq_for_Z1_error} and \eqref{eq_for_Z2_error}, we obtain the equations solved by $Z_{e, i, \infty}$. Evaluating the left upper and the right lower blocks of \eqref{eq_for_Z1_error} and \eqref{eq_for_Z2_error} (or the equations, where $T$ replaced by $\infty$), we can identity the respective blocks of $P_{e, T}$ and $P_e$. Therefore, we have \begin{align}\label{Gramian_blocks}
P_{e, T}=\begin{bmatrix}P_T& P_{2, T}\\ P_{2, T}^\top & {\hat P}_T\end{bmatrix} \quad  \text{and}\quad
P_{e}=\begin{bmatrix}P& P_{2}\\ P_{2}^\top & {\hat P}\end{bmatrix},                                                                                                                                                   \end{align}
where $\hat P$ and $\hat P_T$ are the reachability Gramians of the reduced systems. Moreover, evaluating the right upper block of \eqref{eq_for_Z1_error} and \eqref{eq_for_Z2_error} (or the equations, where $T$ replaced by $\infty$), we have $P_{2, T}:=\sum_{i=1}^\infty\int_0^T Z_{2, i, T}(s)\dd s$ and   $P_2:=\sum_{i=1}^\infty\int_0^\infty Z_{2, i, \infty}(s)\dd s$, where \begin{align*}
 \dot {Z}_{2, 1, T}(t) &= \mathcal L_2\big(Z_{2, 1, T}(t)\big),\quad Z_{2, 1, T}(0) = B\hat B^\top,\\
 \dot {Z}_{2, i, T}(t) &= \mathcal L_2\big(Z_{2, i, T}(t)\big),\quad Z_{2, i, T}(0) = H\big(\sum_{k=1}^{i-1}\int_0^T Z_{2, k, T}(s)\dd s\otimes \int_0^T Z_{2, i-k, T}(s) \dd s\big)\hat H^\top
\end{align*}
with $i\geq 2$ and $\mathcal L_2(X_2)=A X_2 + X_2 \hat A^\top + N (I\otimes X_2) \hat N^\top$. Again, replacing $T$ by $\infty$ leads to the equations for $ Z_{2, i, \infty}$ ($i\geq 1$). Expressing the bound with the Gramian blocks in \eqref{Gramian_blocks}, we obtain \begin{align}\nonumber
\sup_{t\in[0, T]}\left\|y(t)-\hat y(t)\right\|
&\leq  \sqrt{\trace(C P_{T} C^\top)+\trace(\hat C \hat P_{T} \hat C^\top)-2\trace(C P_{2, T} \hat C^\top)}\sqrt{f(u, \mathfrak{u})}\\ \label{general_bound}
&\leq \sqrt{\trace(C P C^\top)+\trace(\hat C \hat P \hat C^\top)-2\trace(C P_{2} \hat C^\top)}\sqrt{f(u, \mathfrak{u})}.
\end{align}
\begin{remark}\label{reamrk_eb_bt}
Let us consider the specific reduced system \eqref{red_sys} with coefficients \eqref{bt_matrices} (balanced truncation). Given $H=0$ (bilinear case), we find a matrix $M$ that is independent of $\Sigma=\diag(\sigma_1, \dots, \sigma_n)$, so that $\trace(C P C^\top)+\trace(\hat C \hat P \hat C^\top)-2\trace(C P_{2} \hat C^\top)=\trace(\Sigma_2 M)$, see \cite{redmannbenner, h2_bil}. Applying similar proof techniques to the general quadratic bilinear case, additional terms occur that do not depend on $\Sigma_2$ explicitly. Therefore, it is an open question if a small matrix $\Sigma_2=\diag(\sigma_{r+1}, \dots, \sigma_n)$ of truncated Hankel singular values ensures that the bound in \eqref{general_bound} is small. We provide further details in the next section.
%We omit all error bound calculation for a  reduced system \eqref{red_sys} with coefficients \eqref{bt_matrices}, since they do not give further insight beyond the above statements of this remark.
\end{remark}

\subsubsection{Specific bound for balanced truncation}\label{BT_bound}

We conduct the specific error analysis solely for the case, where the balancing procedure is based on the pair $(P, Q)$ of infinite Gramians. We introduce an operator describing the quadratic term in \eqref{eq_for_P} and \eqref{eq_for_Q} as follows:\begin{align*}
 \Pi_{H, M}(X):= H\big(M\otimes X\big)H^\top = \sum_{i, j=1}^n H_i X H_j^\top m_{ij}                                                                                                                                                                                                                                                                                                                                                                                                                                                                                              \end{align*}
exploiting the block-wise representation of $H$, where $M=(m_{ij})$ is a suitable matrix. The adjoint  of this operator (w.r.t. the Frobenius inner product) is \begin{align*}
  \Pi_{H, M}^*(Y)= \Pi_{\mathcal H, M^\top}(Y) =  \mathcal H\big(M^\top \otimes Y\big)\mathcal H^\top = \sum_{i, j=1}^n H_i^\top Y H_j m_{ji}.                                                                                                                        \end{align*}
As an immediate consequence, we obtain \begin{align}\nonumber
 \trace(C P C^\top) &= \langle C^\top C, P \rangle_F  = - \langle \mathcal L^*(Q)+ \Pi_{\mathcal H, P}(Q), P \rangle_F = - \langle Q, \mathcal L(P) + \Pi_{H, P}(P) \rangle_F\\ \label{exchange_trace}
 &= \langle Q, BB^\top  \rangle_F = \trace(B^\top Q B)
                                       \end{align}
using \eqref{eq_for_P}, \eqref{eq_for_Q} and $\Pi_{H, P}^*= \Pi_{\mathcal H, P}$ ($P$ is symmetric). Here, we always assume that the Gramians $P$ and $Q$ exist. Without loss of generality, let us now assume that \eqref{original_system} is already balanced, meaning that $P=Q=\Sigma$ leading to \begin{align}
   - B B^\top &=A\Sigma+\Sigma A^\top+\sum_{i=1}^m N_i\Sigma N_i^\top+ \sum_{j=1}^n H_j\Sigma H_j^\top \sigma_j, \\ \label{eq_balanced_Q}
   - C^\top C &=A^\top\Sigma+\Sigma A+\sum_{i=1}^m N_i^\top\Sigma N_i+ \sum_{j=1}^n H_j^\top \Sigma H_j \sigma_j\end{align}
exploiting \eqref{eq_for_P}, \eqref{eq_for_Q} and expressing $N (I\otimes \Sigma) N^\top$, $\mathcal N (I\otimes \Sigma) \mathcal N^\top$, $H (\Sigma\otimes \Sigma) H^\top$, $\mathcal H (\Sigma\otimes \Sigma) \mathcal H^\top$ using the blocks of $N$ and $H$. The evaluation of the first $r$ rows of \eqref{eq_balanced_Q} leads to \begin{align}\nonumber
&- C_1^\top C =\begin{bmatrix}A_{11}^\top & A_{21}^\top\end{bmatrix} \Sigma+ \begin{bmatrix}\Sigma_{1} & 0\end{bmatrix} A+\sum_{i=1}^m \begin{bmatrix}N_{i, 11}^\top & N_{i, 21}^\top\end{bmatrix} \Sigma N_i+ \sum_{j=1}^n \begin{bmatrix}H_{j, 11}^\top & H_{j, 21}^\top\end{bmatrix}\Sigma H_j \sigma_j\\ \nonumber
&=\begin{bmatrix}A_{11}^\top \Sigma_1 & A_{21}^\top\Sigma_2 \end{bmatrix} + \Sigma_{1} \begin{bmatrix}A_{11} & A_{12}\end{bmatrix}+\sum_{i=1}^m N_{i, 11}^\top \Sigma_1 \begin{bmatrix}N_{i, 11} & N_{i, 12}\end{bmatrix}+N_{i, 21}^\top \Sigma_2 \begin{bmatrix}N_{i, 21} & N_{i, 22}\end{bmatrix} \\ \label{first_r_rows_Q}
&\quad%+\sum_{j=1}^n \begin{bmatrix}H_{j, 11}^\top & H_{j, 21}^\top\end{bmatrix}\Sigma H_j \sigma_j
+\sum_{j=1}^n H_{j, 11}^\top \Sigma_1 \begin{bmatrix}H_{j, 11} & H_{j, 12}\end{bmatrix}\sigma_j+H_{j, 21}^\top \Sigma_2 \begin{bmatrix}H_{j, 21} & H_{j, 22}\end{bmatrix}\sigma_j.
\end{align}
If we further select the first $r$ columns of \eqref{first_r_rows_Q}, we have
\begin{equation}\label{left_upper_block_Q}
 \begin{aligned}
- C_1^\top C_1
&=A_{11}^\top \Sigma_1 + \Sigma_{1} A_{11} +\sum_{i=1}^m N_{i, 11}^\top \Sigma_1 N_{i, 11} +N_{i, 21}^\top \Sigma_2 N_{i, 21} \\
& \quad+\sum_{j=1}^n H_{j, 11}^\top \Sigma_1 H_{j, 11}\sigma_j+H_{j, 21}^\top \Sigma_2H_{j, 21} \sigma_j.
\end{aligned}
\end{equation}
Assuming that the reduced Gramians $\hat P=(\hat p_{ij})$ and $\hat Q$ exist, and choosing the reduced system \eqref{red_sys} with coefficients defined in \eqref{bt_matrices}, we obtain \begin{align}\label{eq_red_P}
   - B_1 B_1^\top &=A_{11}\hat P+\hat P A_{11}^\top+\sum_{i=1}^m N_{i, 11}\hat P N_{i, 11}^\top+ \sum_{i, j=1}^r H_{i, 11}\hat P H_{j, 11}^\top {\hat p}_{ij}, \\ \label{eq_red_Q}
   - C_1^\top C_1 &=A_{11}^\top\hat Q+\hat Q A_{11}+\sum_{i=1}^m N_{i, 11}^\top\hat Q N_{i, 11}+ \sum_{i, j=1}^r H_{i, 11}^\top \hat Q H_{j, 11} {\hat p}_{ij}\end{align}
applying \eqref{eq_for_P} and \eqref{eq_for_Q} to the reduced model.
\begin{thm}\label{thm_error_bound_bt}
Let $y$ be the output of \eqref{original_system} that we assume to be balanced, i.e., $P$ as well as $Q$ exist and $Q=P=\Sigma=\diag(\sigma_1, \dots, \sigma_n)$. Suppose that $\hat y$ is the output of the reduced system \eqref{red_sys} with matrices defined in \eqref{bt_matrices} and existing reduced Gramians $\hat P$ and $\hat Q$. We assume that the controls satisfy $f(u, \mathfrak u)<\infty$, where $f$ is introduced in \eqref{def_fuu}. Then, we have
 \begin{align}\label{bound_in_a_theorem}
\sup_{t\in[0, T]}\left\|y(t)-\hat y(t)\right\|
\leq \sqrt{\trace(C \Sigma C^\top)+\trace(C_1 \hat P  C_1^\top)-2\trace(C P_{2} C_1^\top)}\sqrt{f(u, \mathfrak{u})}
\end{align}
with $P_2=\begin{bmatrix}P_{2, 1} \\ P_{2, 2}\end{bmatrix}$  defined in \eqref{Gramian_blocks}. Moreover, we can represent \begin{align}\nonumber                                                                                                                                                &\trace(C \Sigma C^\top)+\trace(C_1 \hat P  C_1^\top)-2\trace(C P_{2} C_1^\top)\\ \label{error_term1}
&= \trace(\Sigma_2 M)                                                                     +\trace\bigg(\Sigma_2 \bigg[\sum_{j=1}^n 2 \begin{bmatrix}H_{j, 21} & H_{j, 22}\end{bmatrix}P_2 H_{j, 21}^\top \sigma_j-\sum_{j=1}^n H_{j, 21}\hat P H_{j, 21}^\top \sigma_j\bigg]\bigg)\\ \label{error_term2}
 &\quad+2\trace\bigg(\Sigma_1 H_1\Big((\Sigma-\begin{bmatrix}P_2 & 0 \end{bmatrix})\otimes\begin{bmatrix}P_2 & 0 \end{bmatrix}\Big)H_1^\top\bigg) - \trace\bigg(\Sigma_1 H_{11}\Big((\Sigma_1-\hat P)\otimes \hat P\Big) H_{11}^\top\bigg),
\end{align}
where $M=B_2 B_2^\top + 2 P_{2, 2} A_{21}^\top + 2\sum_{i=1}^m \begin{bmatrix}N_{i, 21} & N_{i, 22}\end{bmatrix}P_2 N_{i, 21}^\top - \sum_{i=1}^m  N_{i, 21}\hat P N_{i, 21}^\top$ and $H_1$ is the first $r$ rows of $H$.
\end{thm}
\begin{proof}
The bound in \eqref{bound_in_a_theorem} is a result of Section \ref{sec_gen_bound}, see \eqref{general_bound}.
Applying \eqref{exchange_trace} to the reduced system leads to $\trace(C_1 \hat P C_1^\top) = \trace(B_1^\top \hat Q B_1)$. Using this and $\trace(C \Sigma C^\top)=\trace(B^\top \Sigma B)$ for \eqref{bound_in_a_theorem}, we have {\allowdisplaybreaks\begin{align}
\sup_{t\in[0, T]}\left\|y(t)-\hat y(t)\right\|
\leq \sqrt{\trace(B^\top \Sigma B)+\trace(B_1^\top \hat Q  B_1)-2\trace(C P_{2} C_1^\top)}\sqrt{f(u, \mathfrak{u})}.
\end{align}
Inserting \eqref{first_r_rows_Q} into $\trace(C P_{2} C_1^\top)=\trace(P_{2} C_1^\top C)$ yields \begin{align}\nonumber                                                                                                &-\trace(C P_{2} C_1^\top)
\\ \nonumber&= \trace\bigg(P_{2}\bigg[\begin{bmatrix}A_{11}^\top \Sigma_1 & A_{21}^\top\Sigma_2 \end{bmatrix} + \Sigma_{1} \begin{bmatrix}A_{11} & A_{12}\end{bmatrix}+\sum_{i=1}^m N_{i, 11}^\top \Sigma_1 \begin{bmatrix}N_{i, 11} & N_{i, 12}\end{bmatrix}+N_{i, 21}^\top \Sigma_2 \begin{bmatrix}N_{i, 21} & N_{i, 22}\end{bmatrix}\\ \nonumber
&\qquad +\sum_{j=1}^n H_{i, 11}^\top \Sigma_1 \begin{bmatrix}H_{j, 11} & H_{j, 12}\end{bmatrix}\sigma_j+H_{j, 21}^\top \Sigma_2 \begin{bmatrix}H_{j, 21} & H_{j, 22}\end{bmatrix}\sigma_j\bigg]\bigg)  \\ \label{sig1_expr}&= \trace\bigg(\Sigma_1\bigg[P_{2, 1} A_{11}^\top + \begin{bmatrix}A_{11} & A_{12}\end{bmatrix} P_2+\sum_{i=1}^m \begin{bmatrix}N_{i, 11} & N_{i, 12}\end{bmatrix}P_2 N_{i, 11}^\top\bigg]\bigg) \\\nonumber
&\quad + \trace\bigg(\Sigma_1\sum_{j=1}^n \begin{bmatrix}H_{j, 11} & H_{j, 12}\end{bmatrix}\begin{bmatrix} P_2 & 0\end{bmatrix}\begin{bmatrix} H_{j, 11}^\top\\ H_{j, 12}^\top\end{bmatrix}\sigma_j\bigg)
\\ \nonumber
&\quad + \trace\bigg(\Sigma_2 \bigg[P_{2, 2} A_{21}^\top +\sum_{i=1}^m \begin{bmatrix}N_{i, 21} & N_{i, 22}\end{bmatrix}P_2 N_{i, 21}^\top \bigg]\bigg)  +\trace\bigg(\Sigma_2 \sum_{j=1}^n \begin{bmatrix}H_{j, 21} & H_{j, 22}\end{bmatrix}P_2 H_{j, 21}^\top \sigma_j\bigg). \nonumber
\end{align}
Based on \eqref{eq_for_P}, we know that $P_e$ with partition in \eqref{Gramian_blocks} satisfies $- B_e B_e^\top =\mathcal L_e\big(P_e\big)+ H_e\big(P_e\otimes P_e\big)H_e^\top$.
From the right upper block of this equation, we further see that $P_2=(p_{2, ij})$ satisfies \begin{align*}
        - B B_1^\top =A P_2+P_2 A_{11}^\top+\sum_{i=1}^m N_i P_2 N_{i, 11}^\top+ \sum_{i=1}^n\sum_{j=1}^r H_{i}P_2 H_{j, 11}^\top p_{2, ij}.
     \end{align*}
Evaluating the first $r$ rows of this equation yields
\begin{align*}
        &- B_1 B_1^\top =\begin{bmatrix}A_{11} & A_{12}\end{bmatrix} P_2+P_{2, 1} A_{11}^\top+\sum_{i=1}^m \begin{bmatrix}N_{i, 11} & N_{i, 12}\end{bmatrix}  P_2 N_{i, 11}^\top+ \sum_{i=1}^n\sum_{j=1}^r \begin{bmatrix}H_{i, 11} & H_{i, 12}\end{bmatrix} P_2 H_{j, 11}^\top p_{2, ij}\\
&= \begin{bmatrix}A_{11} & A_{12}\end{bmatrix} P_2+P_{2, 1} A_{11}^\top+\sum_{i=1}^m \begin{bmatrix}N_{i, 11} & N_{i, 12}\end{bmatrix}  P_2 N_{i, 11}^\top+ \sum_{i, j=1}^n \begin{bmatrix}H_{i, 11} & H_{i, 12}\end{bmatrix}  \begin{bmatrix}P_2 & 0 \end{bmatrix} \begin{bmatrix}H_{j, 11}^\top\\H_{j, 12}^\top \end{bmatrix} \hat p_{2, ij}
     \end{align*}
using the partition $P_2=\begin{bmatrix}P_{2, 1} \\ P_{2, 2}\end{bmatrix}$ and $\begin{bmatrix}P_2 & 0 \end{bmatrix}=:\hat P_2=(\hat p_{2, ij})$. Inserting this into \eqref{sig1_expr}, we find that \begin{align*}                                                                                                -\trace(C P_{2} C_1^\top)
= &-\trace\bigg(\Sigma_1\bigg[ B_1 B_1^\top+  \sum_{i, j=1}^n \begin{bmatrix}H_{i, 11} & H_{i, 12}\end{bmatrix}  \begin{bmatrix}P_2 & 0 \end{bmatrix} \begin{bmatrix}H_{j, 11}^\top\\H_{j, 12}^\top \end{bmatrix} \hat p_{2, ij}\bigg]\bigg)\\
&+ \trace\bigg(\Sigma_1\sum_{j=1}^n \begin{bmatrix}H_{j, 11} & H_{j, 12}\end{bmatrix}\begin{bmatrix} P_2 & 0\end{bmatrix}\begin{bmatrix} H_{j, 11}^\top\\ H_{j, 12}^\top\end{bmatrix}\sigma_j\bigg)\\
& + \trace\bigg(\Sigma_2 \bigg[P_{2, 2} A_{21}^\top +\sum_{i=1}^m \begin{bmatrix}N_{i, 21} & N_{i, 22}\end{bmatrix}P_2 N_{i, 21}^\top \bigg]\bigg) \\
& +\trace\bigg(\Sigma_2 \sum_{j=1}^n \begin{bmatrix}H_{j, 21} & H_{j, 22}\end{bmatrix}P_2 H_{j, 21}^\top \sigma_j\bigg).
\end{align*}
This identity yields
\begin{equation}\label{second_last_step}
 \begin{aligned}
 &\trace(B^\top \Sigma B)+\trace(B_1^\top \hat Q  B_1)-2\trace(C P_{2} C_1^\top)\\
 &= \trace(B_1^\top \Sigma_1 B_1)+\trace(B_2^\top \Sigma_2 B_2)+\trace(B_1^\top \hat Q  B_1)-2\trace(C P_{2} C_1^\top)\\
 &=\trace(B_1^\top (\hat Q-\Sigma_1)  B_1) + \trace\bigg(\Sigma_2 \bigg[B_2 B_2^\top + 2 P_{2, 2} A_{21}^\top + 2\sum_{i=1}^m \begin{bmatrix}N_{i, 21} & N_{i, 22}\end{bmatrix}P_2 N_{i, 21}^\top \bigg]\bigg)
 \\
&\quad+2\trace\bigg(\Sigma_2 \sum_{j=1}^n \begin{bmatrix}H_{j, 21} & H_{j, 22}\end{bmatrix}P_2 H_{j, 21}^\top \sigma_j\bigg)\\
 &\quad+2\trace\bigg(\Sigma_1 H_1\Big(\Sigma\otimes\begin{bmatrix}P_2 & 0 \end{bmatrix}\Big)H_1^\top\bigg)-2 \trace\bigg(\Sigma_1  H_1\big(\begin{bmatrix}P_2 & 0 \end{bmatrix}\otimes\begin{bmatrix}P_2 & 0 \end{bmatrix}\big)H_1^\top\bigg)
                      \end{aligned}
                      \end{equation}
exploiting that \begin{align*}H_1\big(\Sigma\otimes\begin{bmatrix}P_2 & 0 \end{bmatrix}\big)H_1^\top&=\sum_{j=1}^n \begin{bmatrix}H_{j, 11} & H_{j, 12}\end{bmatrix}\begin{bmatrix} P_2 & 0\end{bmatrix}\begin{bmatrix} H_{j, 11}^\top\\ H_{j, 12}^\top\end{bmatrix}\sigma_j,\\ H_1\big(\begin{bmatrix}P_2 & 0 \end{bmatrix}\otimes\begin{bmatrix}P_2 & 0 \end{bmatrix}\big)H_1^\top&=\sum_{i, j=1}^n \begin{bmatrix}H_{i, 11} & H_{i, 12}\end{bmatrix}  \begin{bmatrix}P_2 & 0 \end{bmatrix} \begin{bmatrix}H_{j, 11}^\top\\H_{j, 12}^\top \end{bmatrix} \hat p_{2, ij}.
\end{align*}
Subtracting \eqref{left_upper_block_Q} from \eqref{eq_red_Q}, we obtain \begin{align}\nonumber
&\sum_{i=1}^m N_{i, 21}^\top \Sigma_2 N_{i, 21} +\sum_{j=1}^n H_{j, 11}^\top \Sigma_1 H_{j, 11}\sigma_j+H_{j, 21}^\top \Sigma_2H_{j, 21} \sigma_j\\ \label{dif_q_sig}
&=A_{11}^\top(\hat Q-\Sigma_{1})+(\hat Q-\Sigma_{1}) A_{11}+\sum_{i=1}^m N_{i, 11}^\top(\hat Q-\Sigma_{1}) N_{i, 11}+\sum_{i, j=1}^r H_{i, 11}^\top \hat Q H_{j, 11} {\hat p}_{ij}.
                  \end{align}
Equations \eqref{eq_red_P} and \eqref{dif_q_sig} lead to \begin{align*}
&\trace(B_1^\top (\hat Q-\Sigma_1)  B_1) = \trace((\hat Q-\Sigma_1)  B_1B_1^\top) \\
&=-\trace\bigg((\hat Q-\Sigma_1) \bigg[ A_{11}\hat P+\hat P A_{11}^\top+\sum_{i=1}^m N_{i, 11}\hat P N_{i, 11}^\top+ \sum_{i, j=1}^r H_{i, 11}\hat P H_{j, 11}^\top {\hat p}_{ij} \bigg]\bigg)\\
&=-\trace\bigg(\hat P\bigg[(\hat Q-\Sigma_1) A_{11}+ A_{11}^\top (\hat Q-\Sigma_1)+\sum_{i=1}^m N_{i, 11}^\top (\hat Q-\Sigma_1)N_{i, 11}+ \sum_{i, j=1}^r H_{i, 11}^\top \hat Q H_{i, 11}{\hat p}_{ij} \bigg]\bigg)\\
&\quad+ \trace\bigg(\Sigma_1 \sum_{i, j=1}^r H_{i, 11}\hat P H_{j, 11}^\top {\hat p}_{ij}\bigg)\\
&=-\trace\bigg(\Sigma_2\sum_{i=1}^m  N_{i, 21}\hat P N_{i, 21}^\top \bigg)-\trace\bigg(\Sigma_2\sum_{j=1}^n  H_{j, 21}\hat P H_{j, 21}^\top \sigma_j \bigg)\\
&\quad+ \trace\bigg(\Sigma_1 \sum_{i, j=1}^r H_{i, 11}\hat P H_{j, 11}^\top {\hat p}_{ij}\bigg)-\trace\bigg(\Sigma_1 \sum_{j=1}^r H_{j, 11}\hat P H_{j, 11}^\top \sigma_j \bigg).
\end{align*}
Inserting this into \eqref{second_last_step} gives us
\begin{align*}
 &\trace(B^\top \Sigma B)+\trace(B_1^\top \hat Q  B_1)-2\trace(C P_{2} C_1^\top)\\
 &= \trace\bigg(\Sigma_2 \bigg[B_2 B_2^\top + 2 P_{2, 2} A_{21}^\top + 2\sum_{i=1}^m \begin{bmatrix}N_{i, 21} & N_{i, 22}\end{bmatrix}P_2 N_{i, 21}^\top - \sum_{i=1}^m  N_{i, 21}\hat P N_{i, 21}^\top \bigg]\bigg)
 \\
&\quad+\trace\bigg(\Sigma_2 \bigg[\sum_{j=1}^n 2 \begin{bmatrix}H_{j, 21} & H_{j, 22}\end{bmatrix}P_2 H_{j, 21}^\top \sigma_j-\sum_{j=1}^n H_{j, 21}\hat P H_{j, 21}^\top \sigma_j\bigg]\bigg)\\
 &\quad+2\trace\bigg(\Sigma_1 H_1\Big((\Sigma-\begin{bmatrix}P_2 & 0 \end{bmatrix})\otimes\begin{bmatrix}P_2 & 0 \end{bmatrix}\Big)H_1^\top\bigg) + \trace\bigg(\Sigma_1 H_{11}\Big((\hat P-\Sigma_1)\otimes \hat P\Big) H_{11}^\top\bigg)
\end{align*}
}
using that $ \sum_{i, j=1}^r H_{i, 11}\hat P H_{j, 11}^\top {\hat p}_{ij}=H_{11}\Big(\hat P\otimes \hat P\Big) H_{11}^\top$ and $\sum_{j=1}^r H_{j, 11}\hat P H_{j, 11}^\top \sigma_j=H_{11}\Big(\Sigma_1\otimes \hat P\Big) H_{11}^\top$.
\end{proof}
The error bound in Theorem \ref{thm_error_bound_bt} consists of terms with an explicit dependence on $\Sigma_2$, see \eqref{error_term1}, and terms whose dependence on $\Sigma_2$ is implicit and whose behavior for small $\Sigma_2$ is not immediately evident, see \eqref{error_term2}. As noted in Remark \ref{reamrk_eb_bt}, the error bound for quadratic bilinear systems appears to be less informative than in the bilinear case ($H=0$). Nevertheless, one important observation is that \eqref{error_term1} contains error terms that depend on products of large and small Hankel singular values. This suggests that the smallness of $\sigma_{r+1}, \dots, \sigma_n$ alone may not be sufficient to guarantee a small error bound. Instead, one may additionally require products such as $\sigma_1 \sigma_{r+1}$ to be small.

\section*{Acknowledgments}
  MR is supported by the DFG via the SFB 1294 ``Data Assimilation''-- subproject A07 ``Data-based model order reduction for stochastic dynamics'' -- Project-ID 318763901.

\bibliographystyle{abbrv}
%\bibliography{ref_quadratic_bilinear_mor}

\end{document}